\documentclass[11pt]{article}

\usepackage{amsmath, amsthm, amssymb, color, enumitem, graphicx}

\usepackage[nobysame, initials]{amsrefs}

\usepackage[draft=false]{hyperref}
\hypersetup{
    colorlinks=true,        
    linkcolor=black,        
    citecolor=black,        
    filecolor=black,        
    urlcolor=black          
}

\usepackage{mathrsfs, pgf, tikz, overpic}
\usepackage[a4paper]{geometry}

\usetikzlibrary{3d, arrows}

\def\itm#1{\rm ({#1})} 
\def\itmit#1{\itm{\it #1\,}} 
\def\rom{\itmit{\roman{*}}}

\def\l{\ell}
\def\eps{\varepsilon}
\def\phi{\varphi}
\def\le{\leqslant}
\def\ge{\geqslant}
\def\CC{\mathbb{C}}

\def\PP{\mathbb{P}}

\def\RR{\mathbb{R}}
\def\SS{\mathbb{S}}
\def\ZZ{\mathbb{Z}}
\newcommand{\setbuilder}[2]{\left\{#1\;\middle|\;#2\right\}}
\newcommand{\tri}{\mathbin{\triangle}}

\newtheorem{theorem}{Theorem}[section]
\newtheorem{lemma}[theorem]{Lemma}
\newtheorem{corollary}[theorem]{Corollary}
\newtheorem{prop}[theorem]{Proposition}

\theoremstyle{definition}
\newtheorem{definition}[theorem]{Definition}
\theoremstyle{remark}

\title{On sets defining few ordinary circles\footnote{This is an update to the published version in Disc.\ Comp.\ Geom.\ \textbf{59} (2018), no.~1, 59--87. Minor errors in Theorems~\ref{thm:orchard} and \ref{thm:strong} have been corrected with corresponding updates to Sections~\ref{constr:ellipse} and \ref{constr:cubic}.}}
\author{Aaron Lin\footnote{Department of Mathematics, London School of Economics and Political Science, United Kingdom.} \and Mehdi Makhul\footnote{Research Institute for Symbolic Computation,
Johannes Kepler University,
Linz, Austria. 
Supported by the Austrian Science Fund
(FWF): DKW1214, subproject DK9.} \and Hossein Nassajian Mojarrad\footnote{Department of Mathematics, EPFL, Lausanne, Switzerland. Partially supported by Swiss National Science Foundation grants 200020-165977 and 200021-162884.} \and Josef Schicho\footnotemark[2] \and Konrad Swanepoel\footnotemark[1] \and Frank de Zeeuw\footnotemark[3]}

\date{}

\begin{document}

\maketitle

\begin{abstract}
An ordinary circle of a set $P$ of $n$ points in the plane is defined as a circle that contains exactly three points of $P$. 
We show that if $P$ is not contained in a line or a circle, then $P$ spans at least $\frac{1}{4}n^2 - O(n)$ ordinary circles. 
Moreover, we determine the exact minimum number of ordinary circles for all sufficiently large $n$ and describe all point sets that come close to this minimum.
We also consider the circle variant of the orchard problem. We prove that 
$P$ spans at most $\frac{1}{24}n^3 - O(n^2)$ circles passing through exactly four points of $P$.
Here we determine the exact maximum and the extremal configurations for all sufficiently large $n$. 

These results are based on the following structure theorem. 
If $n$ is sufficiently large depending on $K$,
and $P$ is a set of $n$ points spanning at most $Kn^2$ ordinary circles, 
then all but $O(K)$ points of $P$ lie on an algebraic curve of degree at most four.
Our proofs rely on a recent result of Green and Tao on ordinary lines, combined with circular inversion and some classical results regarding algebraic curves.
\end{abstract}

\section{Introduction}\label{sec:intro}

\subsection{Background}
The classical Sylvester-Gallai theorem states that any finite non-collinear point set in $\RR^2$ spans at least one \emph{ordinary line} (a line containing exactly two of the points).
A more sophisticated statement is the so-called Dirac-Motzkin conjecture, according to which every non-collinear set of $n>13$ points in $\RR^2$ determines at least $n/2$ ordinary lines. 
This conjecture was proved by Green and Tao \cite{GT13} for all sufficiently large $n$.
Their proof was based on a structure theorem, which roughly states that any point set with a linear number of ordinary lines must lie mostly on a cubic curve 
(see Theorem \ref{thm:GT} for a precise statement).

It is natural to ask the corresponding question for \emph{ordinary circles} (circles that contain exactly three of the given points);
see for instance \cite{BMP05}*{Section~7.2} or \cite{KW91}*{Chapter~6}.
Elliott \cite{E67} introduced this question in 1967, and proved that any $n$ points, not all on a line or a circle, determine at least $\frac{2}{63}n^2-O(n)$ ordinary circles. 
(Throughout the paper, by $O(f(n))$ we mean a function $g(n)$ such that $0\le g(n)\le Cf(n)$ for some constant $C>0$ and all sufficiently large $n$. Thus, $-O(n)$ is a function $g(n)$ satisfying $-Cn\le g(n)\le 0$ for sufficiently large $n$.)
He suggested, cautiously, that the optimal bound is $\frac{1}{6}n^2-O(n)$.
Elliott's result was improved by B\'alintov\'a and B\'alint \cite{BB94}*{Remark, p.~288} to 
$\frac{11}{247}n^2-O(n)$, and Zhang \cite{Z11} obtained $\frac{1}{18}n^2-O(n)$.
Zhang also gave constructions of point sets on two concentric circles with $\frac{1}{4}n^2 - O(n)$ ordinary circles.

We will use the results of Green and Tao to prove that $\frac{1}{4}n^2-O(n)$ is asymptotically the right answer, 
thus disproving the bound suggested by Elliott \cite{E67}.
Nassajian Mojarrad and De Zeeuw proved this bound in an earlier preprint \cite{MZ}, which is subsumed by this paper, and will not be published independently.
We will find the exact minimum number of ordinary circles,
for sufficiently large $n$, and we will determine which configurations attain or come close to that minimum.
We make no attempt to specify the threshold implicit in the phrase `for sufficiently large $n$'; any improvement would depend on an improvement of the threshold in the result of Green and Tao \cite{GT13}.
For small $n$, the bound $\frac{1}{9}\binom{n}{2}$ due to Zhang \cite{Z11} remains the best known lower bound on the number of ordinary circles.

Green and Tao \cite{GT13} also solved (for large $n$) the even older \emph{orchard problem},
which asks for the exact maximum number of lines passing through exactly three points of a set of $n$ points in the plane.
We refer to \cite{GT13} for the history of this problem.
The upper bound $\frac{1}{3}\binom{n}{2}$ is easily proved by double counting, but it is not the exact maximum.
Using group laws on certain cubic curves, one can construct $n$ non-collinear points with $\lfloor\frac{1}{6}n(n-3) + 1\rfloor $ $3$-point lines, 
and Green and Tao \cite{GT13} proved (for large $n$) that this is optimal.
This does not follow directly from the Dirac-Motzkin conjecture, but it does follow from the above-mentioned structure theorem of Green and Tao for sets with few ordinary lines (Theorem \ref{thm:GT}).

The analogous orchard problem for circles asks for the maximum number of circles passing through exactly four points from a set of $n$ points.
As far as we know, this question has not been asked before.
We determine the exact maximum and the extremal sets for all sufficiently large $n$.

Although we do not consider other related problems, we remark that similar questions have been asked for ordinary conics \cite{WW88, CDFGLMSST15, BVZ16}, ordinary planes \cite{B16}, and ordinary hyperplanes \cite{BM16}.

\subsection{Results}\label{sec:results}
Our first main result concerns the minimum number of ordinary circles spanned by a set of $n$ points, not all lying on a line or a circle, and the structure of sets of points that come close to the minimum.
The first part of the theorem solves Problem~6 in \cite{BMP05}*{Section~7.2}.

\begin{theorem}[Ordinary circles]\label{thm:main}
\mbox{}
\begin{enumerate}[label=\rom]
\item\label{parti} If $n$ is sufficiently large, the minimum number of ordinary circles determined by $n$ points in $\RR^2$, not all on a line or a circle, equals
\begin{equation*}
\begin{cases}
\frac14n^2-\frac32n & \text{if } n \equiv 0 \pmod{4},\\
\frac14 n^2 - \frac34n + \frac12 & \text{if } n \equiv 1 \pmod{4},\\
\frac14n^2 - n & \text{if } n \equiv 2 \pmod{4},\\
\frac14 n^2 - \frac54n + \frac32 & \text{if } n \equiv 3 \pmod{4}.
\end{cases}
\end{equation*}
\item Let $C$ be sufficiently large.
If a set $P$ of $n$ points in $\RR^2$ determines fewer than $\frac{1}{2}n^2-Cn$ ordinary circles, 
then $P$ lies on the disjoint union of two circles, or the disjoint union of a line and a circle.
\end{enumerate}
\end{theorem}

In Section~\ref{sec:constr},
we will describe constructions that meet the lower bound in part~\ref{parti} of Theorem \ref{thm:main}.
For even $n$, the bound in part~\ref{parti} is attained by certain constructions on the disjoint union of two circles, 
while for odd $n$, the bound is attained by constructions on the disjoint union of a line and a circle.
The main tools in our proof are circle inversion and the structure theorem of Green and Tao \cite{GT13} for sets with few ordinary lines,
together with some classical results about algebraic curves and their interaction with inversion.

Let us define a \emph{generalised circle} to be either a circle or a line.
Because inversion maps circles and lines to circles and lines, it turns out that in our proof it is more natural to work with generalised circles.
Alternatively, we could phrase our results in terms of the \emph{inversive plane} (or \emph{Riemann sphere}) $\RR^2\cup\{\infty\}$, where $\infty$ is a single point that lies on all lines, 
which can then also be considered as circles.
Yet another equivalent view would be to identify the inversive plane with the sphere $\SS^2$ via stereographic projection, and consider circles on $\SS^2$, which are in bijection with generalised circles.
All our statements about generalised circles in $\RR^2$ could thus be formulated in terms of circles in $\RR^2\cup\{\infty\}$ or on~$\SS^2$.

We define an \emph{ordinary generalised circle} to be one that contains three points from a given set.
Our proof of Theorem \ref{thm:main} proceeds via an analogous theorem for ordinary generalised circles, which turns out to be somewhat easier to obtain.

\begin{theorem}[Ordinary generalised circles]\label{thm:ordgencircles}
\mbox{}
\begin{enumerate}[label=\rom]
\item If $n$ is sufficiently large, the minimum number of ordinary generalised circles determined by $n$ points in $\RR^2$, not all on a generalised circle, equals
\begin{equation*}
\begin{cases}
\frac14 n^2 - n & \text{if } n \equiv 0 \pmod{4},\\
\frac38 n^2 - n + \frac{5}{8} & \text{if } n \equiv 1 \pmod{4},\\
\frac14 n^2 - \frac{1}{2}n & \text{if } n \equiv 2 \pmod{4},\\
\frac38 n^2 - \frac{3}{2}n + \frac{17}{8} & \text{if } n \equiv 3 \pmod{4}.
\end{cases}
\end{equation*}
\item\label{genpartii} Let $C$ be sufficiently large.
If a set $P$ of $n$ points in $\RR^2$ determines fewer than $\frac{1}{2}n^2-Cn$ ordinary generalised circles, 
then
$P$ lies on two disjoint generalised circles.
\end{enumerate}
\end{theorem}

We also solve the analogue of the orchard problem for circles (for sufficiently large $n$).
We define a \emph{$4$-point \textup{(}generalised\textup{)} circle} to be a (generalised) circle that passes through exactly four points of a given set of $n$ points.
The `circular cubics' in part~\ref{genpartii} will be defined in Section \ref{sec:circularcurves}.

\begin{theorem}[$4$-point generalised circles]\label{thm:orchard}
\mbox{}
\begin{enumerate}[label=\rom]
\item  If $n$ is sufficiently large, the maximum number of $4$-point generalised circles determined by a set of $n$ points in $\RR^2$ is equal to
\begin{equation*}
\begin{cases}
\frac{1}{24}n^3 - \frac14 n^2 + \frac{5}{6}n  & \text{if } n \equiv 0 \pmod{8},\\
\frac{1}{24}n^3 - \frac14 n^2 + \frac{11}{24}n - \frac{1}{4}  & \text{if } n \equiv 1, 3, 5, 7 \pmod{8},\\
\frac{1}{24}n^3 - \frac14 n^2 + \frac{7}{12}n - \frac12	& \text{if } n \equiv 2, 6 \pmod{8},\\
\frac{1}{24}n^3 - \frac14 n^2 + \frac56n - 1 & \text{if } n \equiv 4 \pmod{8}.
\end{cases}
\end{equation*}
\item 
Let $C$ be sufficiently large.
If a set $P$ of $n$ points in $\RR^2$ determines more than $\frac{1}{24}n^3 - \frac{7}{24} n^2 + Cn$ $4$-point generalised circles,
then up to inversions, $P$ lies on an ellipse or a smooth circular cubic.
\end{enumerate}
\end{theorem}
Theorem~\ref{thm:orchard} remains true if we replace `generalised circles' by `circles'.
This is because we can apply an inversion to any set of $n$ points with a maximum number of generalised circles in such a way that all straight-line generalised circles become circles.
Therefore, the maximum is also attained by circles only.

The proofs of the above theorems are based on the following structure theorems in the style of Green and Tao \cite{GT13}.
The first gives a rough picture, by stating that a point set with relatively few ordinary generalised circles must lie on a bicircular quartic, a specific type of algebraic curve of degree four that we introduce in Section \ref{sec:circularcurves}.

\begin{theorem}[Weak structure theorem]\label{thm:weak}
Let $K>0$ and let $n$ be sufficiently large depending on $K$.
If a set $P$ of $n$ points in $\RR^2$ spans at most $Kn^2$ ordinary generalised circles,
then all but at most $O(K)$ points of $P$ lie on a bicircular quartic.
\end{theorem}
Ball \cite{B16} concurrently obtained a similar result as a consequence of a structure theorem for ordinary planes in $\RR^3$.
He shows that $n$ points with $O(n^{2+\frac{1}{6}})$ ordinary circles must lie mostly on a quartic curve.

We define bicircular quartics in Section \ref{sec:circularcurves}; they can be reducible, so in Theorem \ref{thm:weak} the set $P$ may also lie mostly on a lower-degree curve contained in a bicircular quartic.
Our proof actually gives a more precise list of possibilities.
The curve that $P$ mostly lies on can be: a line; a circle; an ellipse; a line and a disjoint circle;
two disjoint circles; a circular cubic that is acnodal or smooth;
or a bicircular quartic that is an inverse of an acnodal or smooth circular cubic.

A more precise characterisation of the possible configurations with few ordinary generalised circles is given in the following theorem.
The group structures referred to in the theorem are defined in Section~\ref{sec:groups}; the circular points at infinity ($\alpha$ and $\beta$) referred to in Case~\ref{caseiii} are introduced in Section \ref{sec:circularcurves}; and the `aligned' and `offset' double polygons are defined in Section~\ref{sec:constr}.

\begin{theorem}[Strong structure theorem]\label{thm:strong}
Let $K>0$ and let $n$ be sufficiently large depending on $K$.
If a set $P$ of $n$ points in $\RR^2$ spans at most $Kn^2$ ordinary generalised circles, 
then up to inversions and similarities, $P$ differs in at most $O(K)$ points from a configuration of one of the following types:
\begin{enumerate}[label=\rom]
\item\label{casei} A subset of a line;
\item\label{caseii} A coset $H \oplus x$ of a subgroup of an ellipse, for some $x$ such that $4x \in H$;
\item\label{caseiii} A coset $H \oplus x$ of a subgroup $H$ of a smooth circular cubic, 
for some $x$ such that $4x \in H \oplus \alpha \oplus\beta$, where $\alpha$ and $\beta$ are the two circular points at infinity;
\item\label{caseiv} A double polygon that is `aligned' or `offset'.
\end{enumerate}
Conversely, every set of these types defines at most $O(Kn^2)$ ordinary generalised circles.
\end{theorem}

In Section \ref{sec:circularcurves}, we carefully introduce circular cubics and bicircular quartics, and show their connection to inversion.
In Section \ref{sec:groups}, we define group laws on these curves, which help us construct point sets with few ordinary (generalised) circles in Section \ref{sec:constr}.
In Section \ref{sec:proof}, which forms the core of our proof, we derive Theorems \ref{thm:weak} and \ref{thm:strong} from the structure theorem of Green and Tao \cite{GT13}.
In Section \ref{sec:extremal}, we combine the structure theorems with our analysis of the constructions from Section \ref{sec:constr} to establish the precise statements in Theorems \ref{thm:main}, \ref{thm:ordgencircles}, and \ref{thm:orchard}.


\section{Circular curves and inversion}\label{sec:circularcurves}

The key tool in our proof is \emph{circle inversion}, as it was in the earlier papers \cites{E67, BB94, Z11} on the ordinary circles problem; the first to use circle inversion in Sylvester-Gallai problems was Motzkin \cite{M51}.
The simple reason for the relevance of circle inversion is that if we invert in a point of the given set, an ordinary circle through that point is turned into an ordinary line.
Thus we can use results on ordinary lines, like those of Green and Tao \cite{GT13}, to deduce results about ordinary circles.
To do this successfully, we need a thorough understanding of the effect of inversion on algebraic curves, and in particular we need to introduce the special class of \emph{circular curves}.

\subsection{Circular curves and circular degree}\label{sec:inversion}

In this subsection, we work in the real projective plane $\RR\PP^2$, and partly in the complex projective plane $\CC\PP^2$.
See for instance \cite{ST92}*{Appendix A} for an appropriate introduction to projective geometry.
We use the homogeneous coordinates $[x:y:z]$ for points in $\RR\PP^2$ or $\CC\PP^2$,
and we think of the line with equation $z=0$ as the line at infinity.
An affine algebraic curve in $\RR^2$,
defined by a polynomial $f\in \RR[x,y]$,
can be naturally extended to a projective algebraic curve, 
by taking the zero set of the homogenisation of $f$.
This curve in $\RR\PP^2$ then extends to $\CC\PP^2$, by taking the complex zero set of the homogenised polynomial.

We define the \emph{circular points} to be the points
\[\alpha = [i:1:0],~~~\beta = [-i:1:0] \]
on the line at infinity in $\CC\PP^2$.
The circular points play a key role in this paper, 
due to the fact that every circle contains both circular points. 
Moreover, 
any real conic containing $\alpha$ and $\beta$ is either a circle, or a union of a line with the line at infinity. 
We could thus consider a generalised circle to be a conic that contains both circular points.

\begin{definition}
An algebraic curve in $\RR\PP^2$ is \emph{circular} if it contains $\alpha$ and $\beta$.
For $k\ge 2$, an algebraic curve in $\RR\PP^2$ is \emph{$k$-circular} if it has singularities of multiplicity at least $k$ at both $\alpha$ and $\beta$.
\end{definition}

A classical reference for circular curves is Johnson \cite{J77},
while a more modern one is Werner \cite{W11}.
Let us make the definition more explicit in three concrete cases.

A \emph{generalised circle} is an algebraic curve of degree two that contains $\alpha$ and $\beta$;
equivalently, it is a curve in $\RR\PP^2$ defined by a homogeneous polynomial of the form
\begin{equation*}
t(x^2+y^2) + \ell(x,y,z)z ,
\end{equation*}
where $t\in\RR$, and $\ell\in\RR[x,y,z]$ is a non-trivial linear form.
If $t\neq 0$, then the curve is a circle, while if $t=0$, the curve is the union of a line with the line at infinity.

A \emph{circular cubic} is an algebraic curve of degree three that contains $\alpha$ and $\beta$;
equivalently, it is any curve in $\RR\PP^2$ defined by a homogeneous polynomial of the form
\begin{equation}\label{eq:circular-cubic}
(u x + v y)(x^2+y^2) + q(x,y,z)z ,
\end{equation}
where $u,v\in\RR$, and $q\in \RR[x,y,z]$ is a non-trivial quadratic homogeneous polynomial.
Note that we do not require a circular cubic to be irreducible or smooth.
For instance, the union of a line and a circle is a circular cubic, 
and so is the union of any conic with the line at infinity (take $u=v=0$ in \eqref{eq:circular-cubic}).

A \emph{bicircular quartic} is an algebraic curve of degree four that is $2$-circular; 
equivalently, it is any curve in $\RR\PP^2$ defined by a homogeneous polynomial of the form
\begin{equation}\label{eq:bicircular-quartic}
t(x^2+y^2)^2+ (u x + v y)(x^2+y^2)z + q(x,y,z)z^2,
\end{equation}
where $t,u,v\in\RR$, and $q\in \RR[x,y,z]$ is a non-trivial homogeneous quadratic polynomial (see \cite{W11}*{Section 8.2} for a proof that a quartic is $2$-circular if and only if its equation has the form \eqref{eq:bicircular-quartic}).
A noteworthy example of a bicircular quartic is a union of two circles,
for which it is easy to see that the curve has double points at $\alpha$ and $\beta$, since both circles contain those points.

Every circular cubic is contained in  a bicircular quartic,
since for $t=0$ in \eqref{eq:bicircular-quartic} we get a union of a circular cubic and the line at infinity.
A non-circular conic is also contained in a bicircular quartic, since for $t=u=v=0$ in \eqref{eq:bicircular-quartic} we get a union of a conic and $z^2=0$, which is a double line at infinity.

\begin{definition}\label{def:circular}
The \emph{circular degree} of an algebraic curve $\gamma$ in $\RR\PP^2$ is the smallest $k$ such that $\gamma$ is contained in a $k$-circular curve of degree $2k$.
\end{definition}

The circular degree is well-defined, since given any curve $\gamma$ of degree $k$, we can add $k$ copies of the line at infinity, to get a $k$-circular curve of degree $2k$.

For example, a line has circular degree one, since its union with the line at infinity is a $1$-circular curve of degree two.
A conic that is not a circle has circular degree two, 
since its union with two copies of the line at infinity is a $2$-circular curve of degree four.
Similarly, a circular cubic has circular degree two, since its union with the line at infinity is a $2$-circular curve of degree four.
We can thus classify curves of low circular degree as follows:

\begin{itemize}
\item {\it Circular degree one}: lines and circles (that is, generalised circles);
\item {\it Circular degree two}: non-circular conics, circular cubics, and bicircular quartics;
\item {\it Circular degree three}: non-circular cubics, circular quartics, $2$-circular quintics, and $3$-circular sextics.
\end{itemize}
 
This classification is important to us, because we will see that circular degree is invariant under inversion.

We have defined circular curves and circular degrees in the projective plane, because that is their most natural setting.
In the rest of the paper, to avoid confusion between the projective and inversive planes, 
we will use these notions for curves in $\RR^2$, 
with the understanding that to inspect the definitions we should consider $\RR\PP^2$ and $\CC\PP^2$.

\subsection{Inversion}

Circular curves are intimately related to circle inversion, which we now introduce.
A general reference is \cite{B00}.

\begin{definition}\label{def:inversion}
Let $C(p,r)$ be the circle with centre $p = (x_p, y_p) \in \RR^2$ and radius $r>0$. The \emph{circle inversion} with respect to $C(p,r)$ is the mapping 
$I_{p,r} : \RR^2 \setminus\{p\} \rightarrow \RR^2 \setminus\{p\}$ 
defined by 
\begin{equation*}
I_{p,r}(x,y) = \left( \frac{r^2(x-x_p)}{(x-x_p)^2+(y-y_p)^2} + x_p, \frac{r^2(y-y_p)}{(x-x_p)^2+(y-y_p)^2} + y_p \right)
\end{equation*}
for $(x,y) \ne p$.
We write $I_p$ for $I_{p,1}$.
We call $p$ the \emph{centre} of the inversion $I_{p,r}$.
\end{definition}

In the inversive plane $\RR^2\cup \{\infty\}$, 
the inversion map can be completed by  setting
$I_{p,r}(p) = \infty$ and $I_{p,r}(\infty) = p$,
so that inversions take generalised circles to generalised circles.
The group of transformations of the inversive plane generated by the inversions and the similarities is called the inversive group.
It is known that a bijection of the inversive plane that takes generalised circles to generalised circles has to be an element of this group, and that any element of this group is either a similarity or an inversion followed by an isometry \cite{Co69}*{Theorem~6.71}.

The image of an algebraic curve in $\RR^2$ under an inversion is also an algebraic curve, in the following sense.
\begin{definition}
For any algebraic curve $\gamma$ there is an algebraic curve $\gamma'$ such that 
\[I_{p,r}(\gamma\setminus\{p\}) = \gamma'\setminus\{p\}.\]
We refer to $\gamma'$ as the \emph{inverse of $\gamma$ with respect to the circle $C(p,r)$}, and abuse notation slightly by writing $\gamma'=I_{p,r}(\gamma)$.
Also, since for different choices of radius $r$, $I_{p,r}(\gamma)$ differs only by a dilatation in $p$, we will often only consider the inverse $I_p(\gamma)=I_{p,1}(\gamma)$ and refer to it as the \emph{inverse of $\gamma$ in the point $p$}.
\end{definition}
If a curve has degree $d$, then its inverse has  degree at most $2d$ \cite{W11}*{Theorem 4.14}.
If $\gamma$ is irreducible, then its inverse is also irreducible.
Note that inverses of algebraic curves can behave somewhat unintuitively; for instance, Proposition \ref{prop:ellipseandacnodal} states that the inverse of an ellipse has an isolated point, which is surprising if one thinks of an ellipse as just a closed continuous curve.

It is well known that the inverses of generalised circles are again generalised circles.
It turns out that, more generally,
circular degree is preserved under inversion.
We now make precise what this means for curves of low circular degree.
A proof can be found in the classical paper \cite{J77}; for a more modern reference, see \cite{W11}*{Section 9.2}.

\begin{lemma}[Inversion and circular degree]\label{lem:inversion}
Let $C_k$ be a curve of circular degree $k$. Then:
\begin{enumerate}[label=\rom]
\item 
The inverse of $C_1$ in a point on $C_1$ is a line; the inverse of $C_1$ in a point not on $C_1$ is a circle.
\item\label{inversion-case2}
The inverse of $C_2$ in a singular point on $C_2$ is a non-circular conic; the inverse of $C_2$ in a regular point on $C_2$ is a circular cubic; the inverse of $C_2$ in a point not on $C_2$ is a bicircular quartic.
\item\label{inversion-case3}
The inverse of $C_3$ in a singularity of multiplicity three is a non-circular cubic; 
the inverse of $C_3$ in a singularity of multiplicity two is a circular quartic; 
the inverse of $C_3$ in a regular point on $C_3$ is a $2$-circular quintic; the inverse of $C_3$ in a point not on $C_3$ is a $3$-circular sextic.
\end{enumerate}
\end{lemma}

One particular subcase of Case~\ref{inversion-case2} will play an important role in our paper, and we state it separately in Proposition \ref{prop:ellipseandacnodal}.
A proof can be found in \cite{H20}*{p. 202}.
Let us recall that an \emph{acnodal cubic} is a singular cubic with a singularity that is an isolated point; 
for example, $(2x-1)(x^2+y^2)-y^2=0$ is an acnodal circular cubic with a singularity at the origin.

\begin{prop}\label{prop:ellipseandacnodal}
The inverse of an ellipse in a point on the ellipse is an acnodal circular cubic with the centre of inversion as its singularity; the inverse of an acnodal circular cubic in its singularity is an ellipse through the singularity.
\end{prop}
For example, the inverse of the cubic $(2x-1)(x^2+y^2)-y^2=0$ in its singularity at the origin is the ellipse $(x-1)^2+2y^2=1$.

\section{Groups on circular curves}\label{sec:groups}

\subsection{Groups on irreducible circular cubics}\label{ssec:irreducible}

The extremal configurations in our main theorems are all based on group laws on certain circular curves. 
It is well-known that irreducible smooth cubics (elliptic curves) have a group law (see for instance \cite{ST92}).
These groups play a crucial role in the work of Green and Tao \cite{GT13}.
The reason that these groups are relevant to ordinary lines is the following collinearity property of this group (when defined in the standard way).
Three points on the curve are collinear if and only if in the group they sum to the identity element. 
For this property to hold, the identity element must be an inflection point.
Here we will define a group in a slightly different way (described for instance in \cite{ST92}*{Section 1.2}), in which the identity element is not necessarily an inflection point, and the same collinearity property does not hold. 
However, for circular cubics, we show that we can choose the identity element so that we get a similar property for concyclicity.

First let $\gamma$ be any irreducible cubic, write $\gamma^*$ for its set of regular points, 
and pick an arbitrary point $o\in \gamma^*$.
We describe an additive group operation $\oplus$ on the set $\gamma^*$ for which $o$ is the identity element.
The construction is depicted in Figure~\ref{fig:cir_elliptic}.
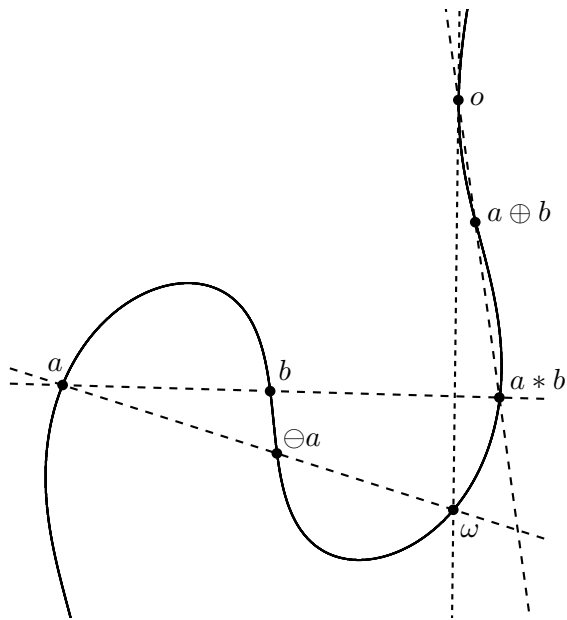
\begin{figure}
\centering
\begin{overpic}[scale=1.0]{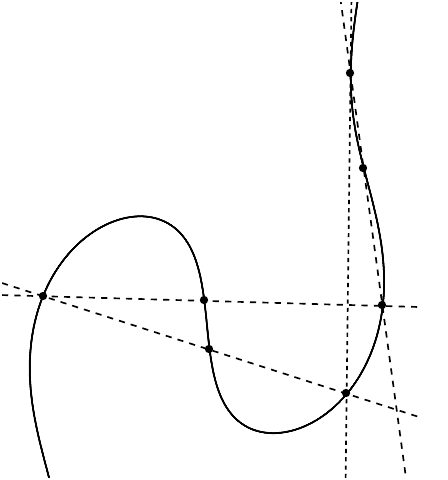}
\put(75,84){$o$}
\put(6.5,40.5){$a$}
\put(44,39){$b$}
\put(81.5,37.5){$a*b$}
\put(77.8,65){$a \oplus b$}
\put(73.5,13.4){$\omega$}
\put(44.7,28.2){$\ominus a$}
\end{overpic}
\caption{Group law on a smooth circular cubic curve}\label{fig:cir_elliptic}
\end{figure}
Given $a,b\in\gamma^*$, 
let $a*b$ be the third intersection point of $\gamma$ and the line $ab$,
and define $a \oplus b$ to be $(a*b)*o$, the third intersection point of $\gamma$ and the line through $a*b$ and $o$. 
When $a=b$, the line $ab$ should be interpreted as the tangent line at $a$; when $a*b = o$, the line through $a*b$ and $o$ should be interpreted as the tangent line to $\gamma$ at $o$.
We refer to \cite{ST92} for a more careful definition and a proof that this operation really does give a group.

Now consider a circular cubic $\gamma$.
Since the circular points $\alpha$ and $\beta$ lying on it are conjugate, 
$\gamma$ has a unique real point on the line at infinity, which we choose as our identity element $o$.
We define the point $\omega$ to be the third intersection point of the tangent line to $\gamma$ at $o$ (if there is no third intersection point, then $o$ is an inflection point, and we consider $o$ itself to be the third point).
Throughout this paper we will use $\omega$ to denote this special point on a circular cubic; note that $\omega$ is not fixed like $\alpha$ and $\beta$, but depends on $\gamma$.
Also note that $\omega$ is real, since it corresponds to the third root of a real cubic polynomial whose other two roots correspond to the real point $o$.
Observe that 
\[\omega = \alpha \oplus \beta,\]
since $\alpha*\beta = o$, and by definition $o*o = \omega$.

With this group law, we no longer have the property that three points are collinear if and only if they sum to $o$ (unless $o$ happens to be an inflection point).
Nevertheless, 
one can check that three points $a,b,c\in\gamma^*$ are collinear if and only if $a\oplus b\oplus c=\omega$.
More important for us, 
four points of $\gamma^*$ lie on a generalised circle if and only if they sum to $\omega$.
This amounts to a classical fact (see \cite{B01}*{Article 225} for an equivalent statement), but we include a proof for completeness.
We use the following version of the Cayley-Bacharach Theorem, due to Chasles (see \cite{EGH96}).

\begin{theorem}[Chasles]\label{thm:cb}
Suppose two cubic curves in $\CC \PP^2$ with no common component intersect in nine points, counting multiplicities. 
If $\gamma$ is another cubic curve containing eight of these intersection points, counting multiplicities, then $\gamma$ also contains the ninth.
\end{theorem}

Recall from Section~\ref{sec:circularcurves} that a generalised circle, viewed projectively, is either a circle, or the union a line with the line at infinity.

\begin{prop}\label{prop:cir_elliptic}
Let $\gamma$ be an irreducible circular cubic in $\RR\PP^2$,
and let $a,b,c,d\in \gamma^*$ be points that are not necessarily distinct.
A generalised circle intersects $\gamma$ in the points $a,b,c,d$ (taking into account multiplicity) if and only if $a\oplus b\oplus c\oplus d=\omega$.
\end{prop}
\begin{proof}
We consider the cubic $\gamma$ extended to $\CC\PP^2$.
We first show the forward direction.
All statements in the proof should be considered with multiplicity.

If the generalised circle is the union of a line $\ell$ and the line at infinity $\ell_\infty$, 
then $\ell\cup \ell_\infty$ intersects $\gamma$ in $a,b,c,d,\alpha,\beta$.
Since $\ell$ intersects $\gamma$ in at most three points,
one of the points $a,b,c,d$ must equal $o$, say $d=o$.
Since $\ell_\infty$ also intersects $\gamma$ in at most three points, we must have $a,b,c\in\ell$.
Thus $a,b,c$ are collinear, and we have $a\oplus b\oplus c = \omega$, by the definition of the group law.
It then follows from $d=o$ that $a\oplus b\oplus c\oplus d=\omega$.

Suppose next that the generalised circle is a circle $\sigma$, 
and intersects $\gamma$ in $a,b,c,d,\alpha,\beta$.
The construction that follows is depicted in Figure~\ref{fig:cb}.
\begin{figure}
\centering
\begin{overpic}[scale=1.0]{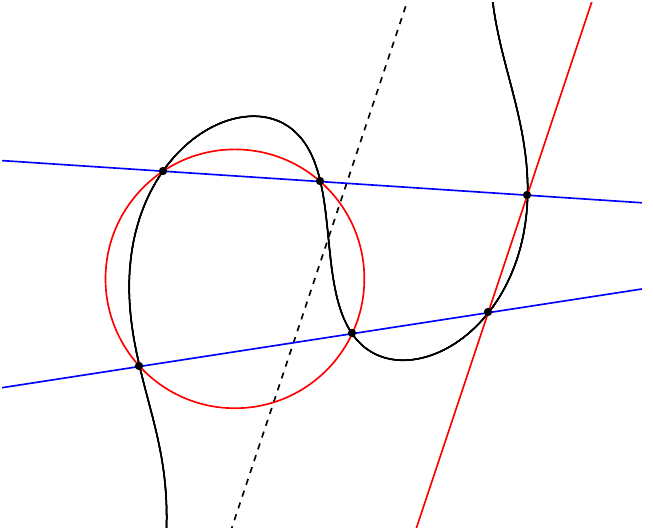}
\put(48,78){{\scriptsize asymptote}}
\put(23,56){$a$}
\put(50,54.5){$b$}
\put(55.5,31.2){$c$}
\put(15,45){$\sigma$}
\put(1,58){$\l_2$}
\put(82.2,52){$a*b$}
\put(86,78){$\l_1$}
\put(75.1,30.5){$a \oplus b$}
\put(1,23.5){$\l_3$}
\put(21.4,27.1){$d'=d$}
\end{overpic}
\caption{Concyclicity of four regular points on a circular cubic}\label{fig:cb}
\end{figure}
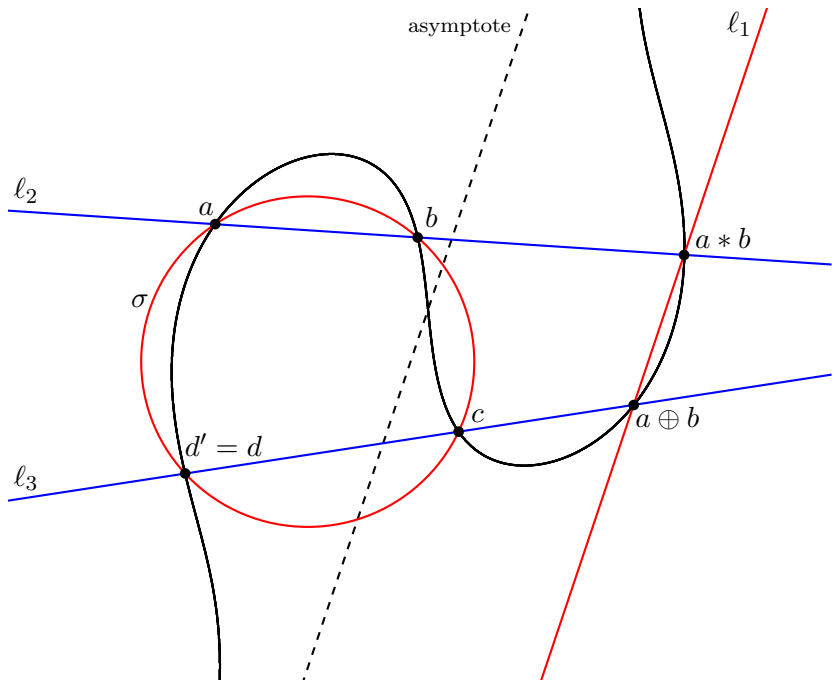
Let $\l_1$ be the line through $o$ and $a*b$ (and thus through $a \oplus b$),  $\l_2$ the line through $a$ and $b$ (and thus through $a*b$), and $\l_3$ the line through $c$ and $a \oplus b$. 
Note that $\sigma$ and $\l_\infty$ intersect in $\alpha$ and $\beta$.
Then $\gamma_1 = \sigma \cup \l_1$ and $\gamma_2 = \l_2 \cup \l_3\cup \l_\infty$ are two cubic curves that intersect in nine points, of which the eight points $a$, $b$, $c$, $a*b$, $a \oplus b$, $o$, $\alpha$, and $\beta$ certainly lie on $\gamma$; the remaining point is the third intersection point of $\gamma_1$ and $\ell_3$ beside $c$ and $a\oplus b$, which we denote by $d'$.
By Theorem~\ref{thm:cb}, $\gamma$ contains $d'$. 
By the group law on $\gamma$, 
we have $d' = (a \oplus b)*c$.
Moreover, $d'$ must be the sixth intersection point of $\gamma$ and $\sigma$ beside $a,b,c,\alpha,\beta$, which is $d$,
so $d = d' =(a \oplus b)*c$.
By the definition of the group law, this implies $a\oplus b\oplus c = o*d$,
so $(a\oplus b\oplus c)* d = (o*d)*d = o$,
and finally $a\oplus b\oplus c\oplus d = o*o = \omega$.

For the converse, suppose that $a\oplus b\oplus c\oplus d=\omega$, and let $d'$ be the fourth point where the generalised circle $\sigma$ through $a,b,c$ intersects $\gamma$.
Then, by what we have just shown, $a\oplus b\oplus c\oplus d'=\omega$, and it follows that $d=d'$, and $a,b,c,d$ lie on $\sigma$.
\end{proof}

This proposition is a consequence of the more general fact that six points on a circular cubic lie on a conic if and only if they sum to $2\omega$.
(In the standard group structure on a cubic, where the identity $o$ is chosen as an inflection point, they would sum to $o$; see \cite{W78}*{Theorem 9.2}.)
Since a generalised circle in $\RR\PP^2$ is a conic containing $\alpha$ and $\beta$,
and $\alpha\oplus \beta = \omega$, 
it follows that four points $a,b,c,d$ (possibly including $o$) lie on a generalised circle if and only if they sum to $\omega$.

\subsection{Groups on other circular curves}
\label{sec:groupsonother}

We now define group laws on two other types of curves of circular degree two, 
and observe that they satisfy similar concyclicity properties.
Let us note at this point that most bicircular quartics can also be given a group structure (if an irreducible bicircular quartic has no singularities besides $\alpha$ and $\beta$, then it is a curve of genus one, and thus has a group law by \cite{S09}*{Section III.3}). 
However, in our proofs we will handle bicircular quartics by inverting in a point on the curve,
which by Lemma \ref{lem:inversion} transforms a bicircular quartic into a circular cubic.
For that reason, we do not need to study the group law on bicircular quartics separately.

\paragraph{Ellipses.}
We discuss a group law on ellipses, 
although we do not actually need it in our proof, 
because inversion lets us transform an ellipse into an acnodal cubic (Proposition \ref{prop:ellipseandacnodal}), which we have already given a group structure in the previous subsection.
Nevertheless, we treat the group law on ellipses here because it is especially elementary, and it would be strange not to mention it.

Consider the ellipse $\sigma$ given by the equation $x^2+ (y/s)^2=1$, with $s\neq 0,1$.
For any point $a\in \sigma$, we project $a$ vertically to the point $a'$ on the unit circle around the origin, 
as in Figure~\ref{fig:ellipse}, and call the angle $\theta_a$ the \emph{eccentric angle} of $a$.
We define the sum of two points $a,b\in \sigma$ to be the point $c=a\oplus b$ whose eccentric angle is $\theta_c =  \theta_a+\theta_b$.
This gives $\sigma$ a group structure isomorphic to $\RR/\ZZ$.
The identity element is $o=(1,0)$, and the inverse of a point is its reflection in the $x$-axis.
We have the following classical fact that describes when four points on an ellipse are concyclic (see \cite{J48} for the oldest reference we could find, and \cite{BPBS84}*{Problem 17.2} for two detailed proofs).

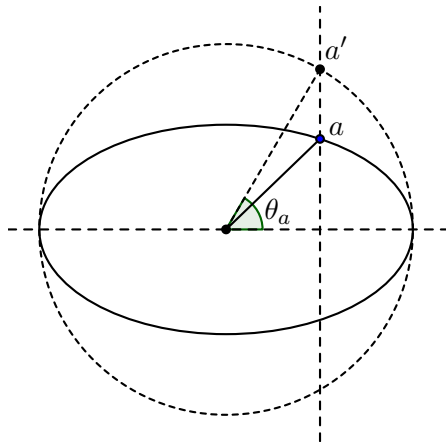
\begin{figure}
\centering
\definecolor{qqwuqq}{rgb}{0.,0.39215686274509803,0.}
\definecolor{qqqqff}{rgb}{0.,0.,1.}
\begin{tikzpicture}[line cap=round,line join=round,>=triangle 45,thick,x=1.0cm,y=1.0cm,scale=1.2]
\clip(-1.5,0) rectangle (3.3,4.8);
\draw [shift={(0.89,2.34)},color=qqwuqq,fill=qqwuqq,fill opacity=0.1] (0,0) -- (0.:0.4) arc (0.:59.807322877532016:0.4) -- cycle;
\draw [rotate around={0.:(0.89,2.34)}] (0.89,2.34) ellipse (2.0480822869644038cm and 1.1569533500436997cm);
\draw [dash pattern=on 3pt off 3pt,domain=-1.5:3.3] plot(\x,{(--7.9092-0.*\x)/3.38});
\draw [dash pattern=on 2pt off 2pt] (0.89,2.34) circle (2.0480822869644038cm);
\draw (0.89,2.34)-- (1.92,3.34);
\draw [dash pattern=on 3pt off 3pt] (1.92,0) -- (1.92,4.8);
\draw [dash pattern=on 2pt off 2pt] (0.89,2.34)-- (1.92,4.110237569982443);
\draw [fill=qqqqff] (1.92,3.34) circle (1.2pt);
\draw[color=black] (2.1,3.45) node {$a$};
\draw [fill=black] (0.89,2.34) circle (1.2pt);
\draw [fill=black] (1.92,4.110237569982443) circle (1.2pt);
\draw[color=black] (2.1,4.35) node {$a'$};
\draw[color=black] (1.46,2.54) node {$\theta_a$};
\end{tikzpicture}
\caption{Eccentric angle of a point on an ellipse}\label{fig:ellipse}
\end{figure}

\begin{prop}\label{prop:ellipse}
Four points $a,b,c,d\in \sigma$ are concyclic if and only if $a\oplus b\oplus c\oplus d =o$.
We may allow two of the points to be equal, in which case the circle through the three distinct points is tangent to the ellipse at the repeated point.
\end{prop}

Another way to look at this group law is that we are parametrising the ellipse using lines through $o=(1,0)$ (see for instance \cite{ST92}*{Section 1.1}).
More precisely, 
each point $a\in \sigma$ corresponds to the line $oa$;
$oa$ makes an angle $\pi-\theta_a/2$ with the $x$-axis,
and the set of lines through $o$ thus has a group structure equivalent to the one above.
This view lets us relate the group on the ellipse to the group on the acnodal cubic.
By Proposition \ref{prop:ellipseandacnodal}, 
inverting in $o$ maps the ellipse to an acnodal circular cubic $\gamma$, with $o$ becoming the isolated point of the cubic.
The lines through $o$ now parametrise the cubic,
and this parametrisation gives the same group on $\gamma$ as the line construction that we gave in Section~\ref{ssec:irreducible} (see \cite{ST92}*{Section 3.7}). 

\paragraph{Concentric circles.}
We now define a group on the union of two disjoint circles.
For notational convenience, we identify $\RR^2$ with $\CC$.
After an appropriate inversion, we can assume the circles to be
\[\sigma_1 = \{ e^{2\pi i t} : t \in [0, 1) \}, ~~~\sigma_2 = \{ re^{-2\pi i t} : t \in [0, 1) \}, \]
with $r>1$,
and we represent each element of $\sigma_1\cup\sigma_2$ as $r^\eps e^{2\pi i t}$ with $\eps\in \ZZ_2$ (with the obvious convention $r^0=1$ and $r^1=r$).
We define a group operation on $\sigma_1\cup\sigma_2$ by
\[r^{\eps_1} e^{2\pi i t_1} \oplus r^{\eps_2} e^{2\pi i t_2} = 
r^{(\eps_1+\eps_2)\bmod 2} e^{2\pi i (t_1+t_2)}, \]
which turns $\sigma_1\cup\sigma_2$ into a group isomorphic to $\RR/\ZZ \times \ZZ_2$,
with identity element $o = 1 = r^0 e^{2\pi i \cdot 0}$.
We again have the following concyclicity property, which is easily seen using symmetry.

\begin{prop}\label{prop:double-polygon-concyclic}
Points $a,b\in \sigma_1$ and  $c,d\in\sigma_2$ lie on a generalised circle if and only if $a\oplus b\oplus c\oplus d =o$.
If $a=b$ or $c=d$, then the generalised circle is tangent at that point.
\end{prop} 


\section{Constructions}\label{sec:constr}

\subsection{Ellipse}\label{constr:ellipse}

Let $\sigma$ be the ellipse defined by $x^2+(y/s)^2 =1$, with the group structure introduced in Section \ref{sec:groupsonother}.
Let $n\ge 5$.
We have a coset of a finite subgroup of size $n$ given by
\[H_n \oplus x=\setbuilder{\left(\cos \left(\frac{2\pi k}{n} - \frac{\pi h}{2n} \right), s\sin \left(\frac{2\pi k}{n} - \frac{\pi h}{2n} \right) \right)}{k=0,\dotsc,n-1} \subset \sigma,\]
where $4x = (\cos(-2\pi h /n), s \sin(-2\pi h/n)) \in H_n$ for some $h \in \{0, \dotsc, n-1\}$.
By Proposition \ref{prop:ellipse}, 
the circle through any three points $a \oplus x,b \oplus x,c \oplus x\in H_n \oplus x$ passes through the point $d \oplus x = \ominus a \ominus b\ominus c \ominus 3x \in H_n \oplus x$.
Therefore, the only way $H_n \oplus x$ spans an ordinary circle is when $d$ coincides with one of $a,b,c$ (which occurs if the circle is tangent to $\sigma$ at that point).
It follows that the number of ordinary circles is equal to 
\[\frac{1}{2}\left|\setbuilder{(k_1,k_2,k_3)\in\ZZ_n^3}{2k_1+k_2+k_3\equiv h \pmod{n}, \quad k_1, k_2, k_3\text{ distinct}}\right|,\]
which is $\frac12 n^2 -O(n)$.

Similarly, the number of $4$-point circles is equal to
\[\frac{1}{4!}\left|\setbuilder{(k_1,k_2,k_3,k_4)\in\ZZ_n^4}{k_1+k_2+k_3+k_4\equiv h \pmod{n}, \quad k_1, k_2, k_3, k_4 \text{ distinct}}\right|,\]
which is, by inclusion-exclusion, equal to $\frac{1}{24}(n^3 - 6n^2 + (8 + 3 \delta_n)n - 6 \eps_n)$, where $\delta_n$ is the number of solutions in $\ZZ_n$ to the equation $2k = h$ and $\eps_n$ is the number of solutions in $\ZZ_n$ to the equation $4k = h$.
Maximising over $h \in \ZZ_n$, this works out to
\begin{equation*}
\begin{cases}
\frac{1}{24}n^3 - \frac14 n^2 + \frac{7}{12}n  & \text{if } n \equiv 0 \pmod{4},\\
\frac{1}{24}n^3 - \frac14 n^2 + \frac{11}{24}n - \frac14  & \text{if } n \equiv 1, 3 \pmod{4},\\
\frac{1}{24}n^3 - \frac14 n^2 + \frac{7}{12}n - \frac12	& \text{if } n \equiv 2 \pmod{4}.
\end{cases}
\end{equation*}

\subsection{Circular cubic curve}\label{constr:cubic}

Let $\gamma$ be an irreducible circular cubic, and let $\oplus$ be the group operation defined in Section \ref{ssec:irreducible}.
It is well known (see for instance \cite{GT13}) that the group $(\gamma^*,\oplus)$ is isomorphic to the circle $\RR/\ZZ$ if $\gamma$ is acnodal or if $\gamma$ is smooth and has one connected component, and is isomorphic to $\RR/\ZZ\times\ZZ_2$ if $\gamma$ is smooth and has two connected components.
Let $H_n$ be a subgroup of order $n$ of $\gamma^*$, and let $x\in\gamma^*$ be such that $4x = \omega\ominus h$ for some $h \in H_n$.
By Proposition \ref{prop:cir_elliptic}, the number of ordinary generalised circles in the coset $H_n \oplus x$ equals
\[\frac{1}{2}\left|\setbuilder{(a, b, c)\in H_n^3}{2a \oplus b \oplus c = h, \quad a, b, c \text{ distinct}}\right|,\]
which is easily seen to equal $\frac12 n^2 - O(n)$.
Similarly, the number of ordinary circles in $H_n \oplus x$ equals
\[\frac{1}{2}\left|\setbuilder{(a, b, c)\in H_n^3}{2a \oplus b \oplus c = h, \quad a, b, c\neq \ominus x \text{ and distinct}}\right|,\]
which also equals $\frac12 n^2 - O(n)$.

As in the previous construction, if $o\notin H_n \oplus x$ (equivalently, $x\notin H_n$) then the number of $4$-point circles is equal to $\frac{1}{24}(n^3 - 6n^2 + (8 + 3 \delta_n)n - 6 \eps_n)$, where $\delta_n$ is the number of solutions in $H_n$ to the equation $2k = h$ and $\eps_n$ is the number of solutions in $H_n$ to the equation $4k = h$. If $H_n$ is cyclic, then we get the same numbers as in the previous construction. Otherwise, $n = 0 \pmod{4}$, $H_n\cong\ZZ_{n/2}\times\ZZ_2$, and maximising over $h \in H_n$, the number of $4$-point circles equals
\begin{equation*}
\begin{cases}
\frac{1}{24}n^3 - \frac14 n^2 + \frac{5}{6}n  & \text{if } n \equiv 0 \pmod{8},\\
\frac{1}{24}n^3 - \frac14 n^2 + \frac{5}{6}n - 1	& \text{if } n \equiv 4 \pmod{8},
\end{cases}
\end{equation*}
which is greater than the corresponding number in the previous construction.

\subsection{`Aligned' double polygons}\label{constr:even}

Let $n\ge 6$ be even and set $m=n/2$.
We identify $\RR^2$ with $\CC$.
Let $\sigma_1$ be the circle with centre the origin and radius one, and $\sigma_2$ the circle with centre the origin and radius $r>1$.
Let $S_1=\setbuilder{e^{2\pi i k/m}}{k=0,\dots,m-1}\subset \sigma_1$ and $S_2=\setbuilder{re^{2\pi i k/m}}{k=0,\dots,m-1}\subset \sigma_2$.
Thus, $S_1$ and $S_2$ are the vertex sets of regular $m$-gons on $\sigma_1$ and $\sigma_2$ that are `aligned' in the sense that their points lie at the same set of angles from the common centre (see Figure \ref{fig:aligned}).

Let $S=S_1\cup S_2$.
By Proposition~\ref{prop:double-polygon-concyclic}, the points $a,b\in\sigma_1$, $c,d\in\sigma_2$ are collinear or concyclic if and only if $a\oplus b\oplus c\oplus d = o$.
In particular, if $a=b$, then the generalised circle through the three points is tangent to $\sigma_1$.
It follows that if $n\ge 8$, the  ordinary generalised circles of $S$ are exactly those through $e^{2\pi i k_1/m}, re^{-2\pi i k_2/m},re^{-2\pi i k_3/m}$ or through $re^{-2\pi i k_1/m}, e^{2\pi i k_2/m},e^{2\pi i k_3/m}$ where $2k_1+k_2+k_3\equiv0\pmod{m}$, with $k_2\not\equiv k_3\pmod{m}$.

For generic $r>1$, we then obtain that the number of ordinary generalised circles equals 
\[\left|\setbuilder{(k_1,k_2,k_3)\in \ZZ_m^3}{2k_1+k_2+k_3\equiv0\pmod{m}, \quad k_2, k_3 \text{ distinct}}\right|\]
(although $k_2$ and $k_3$ are not ordered, we either have two points on $\sigma_1$ or two points on $\sigma_2$).
This equals $m(m-2)$ if $m$ is even and $m(m-1)$ if $m$ is odd.
That is, for generic $r$, we obtain $\frac{1}{4}n^2-n$ ordinary generalised circles if $n\equiv 0\pmod{4}$ and $\frac{1}{4}n^2-\frac12n$ ordinary generalised circles if $n\equiv 2\pmod{4}$. 

If we choose $r=(\cos(2\pi k/m))^{-1}$ (there are $\lfloor m/4\rfloor$ choices for $r$), then the tangent lines at points of $S_1$ pass through two points of $S_2$, so are ordinary generalised circles.
Thus, for these choices of $r$ we lose $m$ ordinary circles, and obtain $\frac{1}{4}n^2-\frac{3}{2}n$ ordinary circles if $n\equiv 0\pmod{4}$ and $\frac{1}{4}n^2-n$ ordinary circles if $n\equiv 2\pmod{4}$.
Note that this is much less than the number of ordinary circles given by Constructions~\ref{constr:ellipse} and~\ref{constr:cubic}.
 
Similarly, the number of $4$-point generalised circles spanned by $S$ equals
\[\frac{1}{4}\left|\setbuilder{(k_1,k_2,k_3,k_4)\in \ZZ_m^4}{k_1+ k_2+ k_3+k_4\equiv0\pmod{m}, \quad k_1 \ne k_2 \text{ and } k_3 \ne k_4}\right|,\]
which is $\frac14 m^3 - O(m^2) = \frac{1}{32}n^3 - O(n^2)$, also much less than the number in Constructions~\ref{constr:ellipse} and~\ref{constr:cubic}.

\begin{figure}
\centering
\begin{minipage}{0.4\textwidth}
\centering
\definecolor{cqcqcq}{rgb}{0.0,0.0,0.6}
\definecolor{qqqqff}{rgb}{0.,0.,0.}
\begin{tikzpicture}[line cap=round,line join=round,>=triangle 45,thick,x=1.0cm,y=1.0cm,scale=1.25]
\clip(-2.2,-2.2) rectangle (2.2,2.2);
\draw(0.,0.) circle (2.cm);
\draw(0.,0.) circle (1.cm);
\draw [color=cqcqcq] (0.49422649730810375,0.8560254037844387) circle (0.9942767863877447cm);
\draw [color=cqcqcq] (-0.3,-0.5196152422706631) circle (1.4cm);
\draw [fill=qqqqff] (2.,0.) circle (1.4pt);
\draw [fill=qqqqff] (1.,1.7120508075688772) circle (1.4pt);
\draw [fill=qqqqff] (-1.,1.7320508075688772) circle (1.4pt);
\draw [fill=qqqqff] (-2.,0.) circle (1.4pt);
\draw [fill=qqqqff] (-1.,-1.7320508075688772) circle (1.4pt);
\draw [fill=qqqqff] (1.,-1.7320508075688772) circle (1.4pt);
\draw [fill=qqqqff] (1.,0.) circle (1.4pt);
\draw [fill=qqqqff] (0.5,0.8660254037844386) circle (1.4pt);
\draw [fill=qqqqff] (-0.5,0.8660254037844386) circle (1.4pt);
\draw [fill=qqqqff] (-1.,0.) circle (1.4pt);
\draw [fill=qqqqff] (-0.5,-0.8660254037844386) circle (1.4pt);
\draw [fill=qqqqff] (0.5,-0.8660254037844386) circle (1.4pt);
\draw [color=black] (-1,-.7) node {$\sigma_1$};
\draw [color=black] (1.7,-.5) node {$\sigma_2$};
\end{tikzpicture}
\caption{`Aligned' double hexagon}\label{fig:aligned}
\end{minipage}
\qquad
\begin{minipage}{0.4\textwidth}
\centering
\definecolor{cqcqcq}{rgb}{0.0,0.0,0.6}
\definecolor{qqqqff}{rgb}{0.,0.,0.}
\begin{tikzpicture}[line cap=round,line join=round,>=triangle 45,thick,x=1.0cm,y=1.0cm,scale=1.25, rotate=90]
\clip(-2.2,-2.2) rectangle (2.2,2.2);
\draw(0.,0.) circle (2.cm);
\draw(0.,0.) circle (1.cm);
\draw [color=cqcqcq] (0.6562278325107631,1.1366199472494412) circle (0.6702984246947132cm);
\draw [color=cqcqcq] (-0.26168644528051416,-0.45325421887794337) circle (1.4766271094389718cm);
\draw [color=cqcqcq] (0.375,-0.6495190528383289) circle (1.25cm);
\draw [color=cqcqcq] (-0.375,0.6495190528383289) circle (1.25cm);
\draw [fill=qqqqff] (2.,0.) circle (1.4pt);
\draw [fill=qqqqff] (1.,1.7120508075688772) circle (1.4pt);
\draw [fill=qqqqff] (-1.,1.7320508075688772) circle (1.4pt);
\draw [fill=qqqqff] (-2.,0.) circle (1.4pt);
\draw [fill=qqqqff] (-1.,-1.7320508075688772) circle (1.4pt);
\draw [fill=qqqqff] (1.,-1.7320508075688772) circle (1.4pt);
\draw [fill=qqqqff] (0.,1.) circle (1.4pt);
\draw [fill=qqqqff] (0.8660254037844386,0.5) circle (1.4pt);
\draw [fill=qqqqff] (0.8660254037844386,-0.5) circle (1.4pt);
\draw [fill=qqqqff] (0.,-1.) circle (1.4pt);
\draw [fill=qqqqff] (-0.8660254037844386,0.5) circle (1.4pt);
\draw [fill=qqqqff] (-0.8660254037844386,-0.5) circle (1.4pt);
\draw [color=black] (.5,-1.1) node {$\sigma_1$};
\draw [color=black] (1.7,.5) node {$\sigma_2$};
\end{tikzpicture}
\caption{`Offset' double hexagon}\label{fig:offset}
\end{minipage}
\end{figure}

\subsection{`Offset' double polygons}\label{constr:offset}
We modify the previous construction by rotating $S_2$ around the origin by an angle of $\pi k/m$.
This results in $S_2'=\setbuilder{re^{-i \pi (2k-1)/m}}{k=0,\dots,m-1}$ and $S'=S_1\cup S_2'$ (see Figure~\ref{fig:offset}). 
As before, if $n\ge 8$, the ordinary generalised circles of $S'$ are exactly those through $e^{2\pi i k_1/m}, re^{-i\pi (2k_2-1)/m},re^{-i\pi (2k_3-1)/m}$ or through $re^{-i\pi (2k_1-1)/m}, e^{2\pi i k_2/m},e^{2\pi i k_3/m}$, where $2k_1+k_2+k_3\equiv1\pmod{m}$ with $k_2\not\equiv k_3\pmod{m}$.

For generic $r>1$, we now have to count the number of ordered triples in the set 
\[\setbuilder{(k_1,k_2,k_3)\in \ZZ_m^3}{2k_1+k_2+k_3\equiv1\pmod{m}, \quad k_2, k_3 \text{ distinct}}.\]
This equals $m^2$ if $m$ is even and $m(m-1)$ if $m$ is odd.
That is, for generic $r$, we obtain $\frac{1}{4}n^2$ ordinary generalised circles if $n\equiv 0\pmod{4}$, worse than Construction~\ref{constr:even}, and $\frac{1}{4}n^2-\frac12n$ ordinary generalised circles if $n\equiv2\pmod{4}$, the same number as in Construction~\ref{constr:even}.

Again, if we choose $r=(\cos(2\pi k/m))^{-1}$ (there are $\lfloor m/4\rfloor$ choices for $r$), we lose $m$ ordinary circles.
Thus, we obtain $\frac{1}{4}n^2-n$ ordinary circles if $n\equiv2\pmod{4}$, the same number as in Construction~\ref{constr:even}.

As in Construction~\ref{constr:even}, we get $\frac{1}{32}n^3 - O(n^2)$ $4$-point circles.

\subsection{Punctured double polygons}\label{constr:odd}
Let $n=2m-1\ge 11$ be odd. Take Construction~\ref{constr:even} with $n+1=2m$ points and remove an arbitrary point $p\in S_1$.

First assume that $m$ is odd.
Before we remove $p$, there are $m(m-1)$ ordinary generalised circles.
Of these, there are $(m-1)/2$ tangent at $p$. There are also $m-1$ ordinary generalised circles through $p$ tangent at some point of $S_2$.
Thus, by removing $p$, we destroy $3(m-1)/2$ ordinary generalised circles and create $\binom{m}{2}-(m-1)/2$ new ones.
Therefore, $S\setminus\{p\}$ has
\[m(m-1)-\frac{3}{2}(m-1)+\binom{m}{2}-\frac{1}{2}(m-1) = \frac32 m^2 - \frac{7}{2}m + 2\] ordinary generalised circles.
That is, there are $\frac38 n^2 - n + \frac{5}{8}$ ordinary generalised circles if $n\equiv1\pmod{4}$.

Next assume that $m$ is even.
Before we remove $p$, there are $m(m-2)$ ordinary generalised circles, of which 
there are $(m-2)/2$ through two different points of $S_2$ tangent at $p$, and there are also $m-2$ ordinary generalised circles through $p$ tangent at a point of $S_2$.
As before, we obtain
\[ m(m-2)-\frac32(m-2)+\binom{m}{2}-\frac12(m-2) = \frac32 m^2 - \frac{9}{2}m + 4\] ordinary generalised circles.
Thus, we obtain $\frac38 n^2 - \frac{3}{2}n + \frac{17}{8}$ ordinary generalised circles if $n\equiv3\pmod{4}$.

Instead of starting with Construction~\ref{constr:even}, we can take the `offset' Construction~\ref{constr:offset} and remove a point.
It is easy to see that when $n\equiv1\pmod{4}$ we obtain the same number of ordinary generalised circles, while if $n\equiv3\pmod{4}$ we obtain more.

Since there are no $5$-point circles in Constructions~\ref{constr:even} and \ref{constr:offset} when $m\ge 6$, removing a point does not add any $4$-point circle, but destroys $O(n^2)$ of them. We thus get $\frac{1}{32}n^3 - O(n^2)$ $4$-point generalised circles, which is asymptotically the same as in Constructions~\ref{constr:even} and \ref{constr:offset}.

\subsection{Inverted double polygons}\label{constr:linecircle}
We can use inversion to make new constructions out of old ones.

Invert Construction~\ref{constr:odd} in the removed point $p$.
The resulting point set has $m$ points on a circle and $m-1$ points on a line disjoint from the circle.
Every ordinary circle after the inversion corresponds to an ordinary generalised circle not passing through $p$ before the inversion.
If $m$ is odd, there are $(m-1)/2$ ordinary generalised circles tangent at $p$ and a further $m-1$ ordinary generalised circles through $p$ tangent to $\sigma_2$,
so we obtain $m(m-1) - 3(m-1)/2 = \frac12(m-1)(2m-3)$ ordinary circles.
For even $m$ we similarly obtain $m(m-2) - 3(m-2)/2 = \frac12(m-2)(2m-3)$ ordinary circles.
That is, we have $\frac14(n-1)(n-2) = \frac14 n^2 - \frac{3}{4}n + \frac12$ ordinary circles when $n\equiv1\pmod{4}$ and $\frac14(n-3)(n-2) = \frac14 n^2 - \frac{5}{4}n + \frac32$ ordinary circles when $n\equiv3\pmod{4}$.

If we remove another point from this inverted construction,  we obtain a set of $n$ points where $n$ is even, with $\frac38n^2 - O(n)$ ordinary circles.

\subsection{Other inverted examples}
If we invert Construction~\ref{constr:ellipse} in a point on the ellipse that is not in the set $S$, then by Proposition~\ref{prop:ellipseandacnodal}, we obtain points on an acnodal circular cubic (without its acnode) as in Construction~\ref{constr:cubic}, with the same number of ordinary and $4$-point generalised circles.

If we invert a circular cubic in a point not on the curve, then we obtain a bicircular quartic by Lemma~\ref{lem:inversion}. 
There will again be $\frac12 n^2 - O(n)$ ordinary circles (or ordinary generalised circles) and $\frac{1}{24}n^3-O(n^2)$ $4$-point circles among the inverted points.


\section{The structure theorems}\label{sec:proof}

\subsection{Proof of the weak structure theorem}

The proofs of our structure theorems for sets with few ordinary circles crucially rely on the following structure theorem for sets with few ordinary lines due to Green and Tao~\cite{GT13}.
Recall that an \emph{ordinary line} is a line containing exactly two points of the given point set.

\begin{theorem}[Green--Tao]\label{thm:GT}
Let $K>0$ and let $n$ be sufficiently large depending on $K$.
If a set $P$ of $n$ points in $\RR^2$ spans at most $Kn$ ordinary lines, then $P$ differs in at most $O(K)$ points from an example of one of the following types:
\begin{enumerate}[label=\rom]
\item $n - O(K)$ points on a line;
\item $m$ points each on a line and a disjoint conic, for some $m = n/2 \pm O(K)$;
\item $n \pm O(K)$ points on an acnodal or smooth cubic.
\end{enumerate}
\end{theorem}

\bigskip

We commence the proof of Theorem~\ref{thm:weak}.
Let $P$ be a set of $n$ points spanning at most $Kn^2$ ordinary generalised circles.
We wish to show that $P$ lies mostly on a bicircular quartic (we will repeatedly use `mostly' to mean `for all but $O(K)$ points').

Note that for at least $2n/3$ points $p$ of $P$, there are at most $9Kn$ ordinary circles through $p$, hence the set $I_p(P \setminus \{p\})$ spans at most $9Kn$ ordinary lines. 
Let $P'$ be the set of such points.
For $n$ sufficiently large depending on $K$,
applying Theorem~\ref{thm:GT} to $I_p(P \setminus \{p\})$ for any $p \in P'$ gives that $I_p(P \setminus \{p\})$ lies mostly on a line, a line and a conic, an acnodal cubic, or a smooth cubic.

If there exists $p \in P'$ such that $I_p(P \setminus \{p\})$ lies mostly on a line, then inverting again in $p$, we see that $P$ must lie mostly on a line or a circle.

If there exists $p \in P'$ such that $I_p(P \setminus \{p\})$ lies mostly on a line $\l$ and a disjoint conic $\sigma$, 
we have two cases,
depending on whether $p$ lies on $\l$ or not.

If $p \in \l$, we invert again in $p$ to find that $P$ lies mostly on the union of $\l$ and $I_p(\sigma)$.
By Lemma~\ref{lem:inversion}, 
$I_p(\sigma)$ is either a circle (if $\sigma$ is a circle) or an irreducible bicircular quartic (if $\sigma$ is a non-circular conic). 
Furthermore, $p$ is the only point that could possibly lie on both $\l$ and $I_p(\sigma)$. 
Since roughly $n/2$ points of $P$ lie on $\l$, there must be another point $q \in \l \cap P'$ that does not lie on $I_p(\sigma)$.
In $I_q(P \setminus \{q\})$, 
the line $\l$ remains a line, 
and by definition of $P'$ the set $I_q(P \setminus \{q\})$ spans few ordinary lines,
so Theorem~\ref{thm:GT} tells us $I_q(I_p(\sigma))$ is a conic. 
It follows from Lemma~\ref{lem:inversion} that $I_p(\sigma)$ cannot be a quartic, 
since we inverted in the point $q$ outside $I_p(\sigma)$ and did not obtain a quartic.
That means $I_p(\sigma)$ has to be a circle, and it is disjoint from $\l$.
Thus, $P$ lies mostly on the union of a line and a disjoint circle.

If $p \notin \l$, we invert in $p$ to see that $P$ lies mostly on the union of the circle $I_p(\l)$ and the curve $I_p(\sigma)$, which is either a circle or a quartic.
Again $p$ is the only point that can lie on both curves.
Inverting in another point $q \in I_p(\l) \cap P'$,
$I_q(I_p(\l))$ becomes a line, 
so Theorem~\ref{thm:GT} tells us that $I_q(I_p(\sigma))$ is a conic, so that $I_p(\sigma)$ must be a circle disjoint from $I_p(\l)$
as before.
Thus, $P$ lies mostly on the union of two disjoint circles.

The case that remains is when for all $p\in P'$, the set $I_p(P \setminus \{p\})$ lies mostly on an acnodal or smooth cubic $\gamma$. 
Fix such a $p$, and consider $I_p(\gamma)$, which mostly contains $P$. 
If $\gamma$ is not a circular cubic, then by the classification in Section \ref{sec:circularcurves} it has circular degree three, 
so $I_p(\gamma)$ has circular degree three as well. 
For any $q \in I_p(\gamma) \cap P'$ other than $p$,
the curve $I_q(I_p(\gamma))$ is also a cubic curve,
 by the definition of $P'$ and Theorem \ref{thm:GT}.
By Case~\ref{inversion-case3} of Lemma~\ref{lem:inversion},
this can only happen if $q$ is a singularity of $I_p(\gamma)$.
But $I_p(\gamma)$ is an irreducible curve of degree at most six, and so has at most 10 singularities by~\cite{W78}*{Theorem 4.4}, which is a contradiction.
So $\gamma$ must be a circular cubic that is acnodal or smooth.
If $\gamma$ is acnodal, then $I_p(\gamma)$ is either a bicircular quartic (if $p\not\in \gamma$), 
an acnodal circular cubic (if $p$ is a regular point of $\gamma$), 
or a non-circular conic (if $p$ is the singularity of $\gamma$).
In the last case, the conic is an ellipse by Proposition \ref{prop:ellipseandacnodal}.
If $\gamma$ is smooth, 
then $I_p(\gamma)$ is either a bicircular quartic or a smooth circular cubic.

We have encountered the following curves that $P$ could mostly lie on: 
a line, a circle, an ellipse, a disjoint union of a line and a circle, a disjoint union of two circles, a circular cubic, or a bicircular quartic.
All of these are subsets of bicircular quartics, which proves the statement of Theorem \ref{thm:weak}.
\hfill$\blacksquare$

\subsection{Proof of the strong structure theorem}
We now prove Theorem~\ref{thm:strong}.
First of all, as explained in Section~\ref{sec:constr}, a subgroup of an ellipse and an appropriate coset of a subgroup of a smooth circular cubic both have at most $\frac12 n^2$ ordinary generalised circles, and a double polygon has at most $\frac14 n^2$ ordinary generalised circles.
It follows from Lemma~\ref{lemma:stability} below that if we add and/or remove $O(K)$ points, then there will be at most $O(Kn^2)$ ordinary generalised circles.

\begin{lemma}\label{lemma:stability}
Let $S$ be a set of $n$ points in $\RR^2$ with $s$ ordinary generalised circles.
Let $T$ be a set that differs from $S$ in at most $K$ points: $|S\tri T|\le K$.
Then $T$ has at most $s + O(Kn^2 + K^2n + K^3)$
ordinary generalised circles.
\end{lemma}
\begin{proof}
First note that if we add a point to any set of $n$ points, we create at most $\binom{n}{2}$ ordinary generalised circles.
Secondly, since two circles intersect in at most two points, the number of $4$-point circles through a fixed point in a set of $n$ points is at most $\frac13\binom{n-1}{2}$, so by removing a point we create at most $\frac13\binom{n-1}{2}<\binom{n}{2}$ ordinary generalised circles.
It follows that by adding and removing $O(K)$ points, we create at most
\[ \binom{n}{2}+\binom{n+1}{2}+\dots+\binom{n+K-1}{2} =O(Kn^2 + K^2n + K^3)
\]
ordinary generalised circles.
\end{proof}

Next, let $P$ be a set of $n$ points with at most $Kn^2$ ordinary generalised circles.
From the proof of Theorem~\ref{thm:weak} above, we see that  $P$ differs in at most $O(K)$ points from a line, a circle, an ellipse, a disjoint union of a line and a circle, a disjoint union of two circles, a circular cubic, or a bicircular quartic. 
Moreover, in the proof we saw that the circular cubic must be acnodal or smooth, and that the bicircular quartic has the property that if we invert in a point on the curve, the resulting circular cubic is acnodal or smooth.

Using inversions, we can reduce the number of types of curves that we need to analyse further.

\begin{itemize}
\item If $P$ lies mostly on a line, then we are in Case~\ref{casei} of Theorem~\ref{thm:strong}, so we are done. 

\item If $P$ lies mostly on a circle, then inverting in a point on the circle puts us in Case~\ref{casei} again. 

\item If $P$ lies mostly on an ellipse, 
then inverting in a point of the ellipse places $P$ mostly on an acnodal circular cubic.

\item If $P$ lies mostly on a bicircular quartic, then inverting in any regular point on the curve gives us a circular cubic.
As mentioned above, this cubic is acnodal or smooth.

\item If $P$ lies mostly on a line and a disjoint circle, 
then an inversion in a point not on the line or circle places $P$ mostly on two disjoint circles.

\item If $P$ lies mostly on the disjoint union of two circles, 
we can apply an inversion that maps the two disjoint circles to two concentric circles \cite{B00}*{Theorem 1.7}.

\end{itemize}

So, up to inversions, we need only consider the cases when $P$ lies mostly on an acnodal or smooth circular cubic, 
or on two concentric circles.
We do this in Lemmas~\ref{lem:cir_elliptic} and \ref{lem:2cir} below, which will complete the proof of Theorem~\ref{thm:strong}.

To determine the structure of $P$,
we use a variant of a lemma from additive combinatorics that was used by Green and Tao \cite{GT13}.
It captures the principle that if a finite subset of a group is almost closed under addition, then it is close to a subgroup.
The following statement is Proposition A.5 in \cite{GT13}.

\begin{prop}\label{prop:A5}
Let $K>0$ and let $n$ be sufficiently large depending on $K$.
Let $A$, $B$, $C$ be three subsets of some abelian group $(G,\oplus)$, all of cardinality within $K$ of $n$. Suppose there are at most $Kn$ pairs $(a,b) \in A \times B$ for which $a \oplus b \notin C$. Then there is a subgroup $H \le G$ and cosets $H \oplus x$, $H \oplus y$ such that
\begin{equation*}
|A \tri (H \oplus x)|, |B \tri (H \oplus y)|, |C \tri (H \oplus x \oplus y)| = O(K).
\end{equation*}
\end{prop}

The variant that we need is a simple corollary of Proposition~\ref{prop:A5}.

\begin{corollary}\label{cor:four}
Let $K>0$ and let $n$ be sufficiently large depending on $K$.
Let $A$, $B$, $C$, $D$ be four subsets of some abelian group $(G,\oplus)$, all of cardinality within $K$ of $n$.
Suppose there are at most $Kn^2$ triples $(a,b,c) \in A \times B \times C$ for which $a \oplus b \oplus c \notin D$. Then there is a subgroup $H \le G$ and cosets $H \oplus x$, $H \oplus y$, $H \oplus z$ such that 
\begin{equation*}
|A \tri (H \oplus x)|, |B \tri (H \oplus y)|, |C \tri (H \oplus z)|, |D \tri (H \oplus x \oplus y \oplus z)| = O(K).
\end{equation*}
\end{corollary}

\begin{proof}
By the pigeonhole principle, there exists an $a_0 \in A$ such that there are at most $K'n$ (where $K' = O(K)$) pairs $(b,c) \in B \times C$ for which $a_0 \oplus b \oplus c \notin D$, or equivalently $b\oplus c\notin D\ominus a_0$. 
Applying Proposition~\ref{prop:A5}, we have a subgroup $H \le G$ and cosets $H \oplus y$, $H \oplus z$ such that
\begin{equation*}
|B \tri (H \oplus y)|, |C \tri (H \oplus z)|, |(D \ominus a_0) \tri (H \oplus y \oplus z)| = O(K).
\end{equation*}

Since $|B\cap(H\oplus y)|\ge n-O(K)$, we repeat the argument above to obtain $b_0 \in B\cap(H\oplus y)$ such that there are at most $O(Kn)$ pairs $(a,c)\in A\times C$ with $a\oplus b_0\oplus c\notin D$, and Proposition~\ref{prop:A5} gives a subgroup $H' \le G$ and cosets $H' \oplus x$, $H' \oplus z'$ such that
\begin{equation*}
|A \tri (H' \oplus x)|, |C \tri (H' \oplus z')|, |(D \ominus b_0) \tri (H' \oplus x \oplus z')| = O(K).
\end{equation*}

From this, it follows that $|(H \oplus z) \tri (H' \oplus z')| = O(K)$, hence $|(H \oplus z) \cap (H' \oplus z')| \ge n - O(K)$.
Since $(H\oplus z)\cap(H'\oplus z')$ is not empty, it has to be a coset of $H' \cap H$.
If $H'\neq H$, then $|H'\cap H| \le n/2 + O(K)$, a contradiction.
Therefore, $H=H'$ and $H \oplus z = H' \oplus z'$.
So we have $|A \tri (H \oplus x)|, |B \tri (H \oplus y)|, |C \tri (H \oplus z)|, |D\tri(H\oplus x\oplus b_0\oplus z)| = O(K)$.
Since $b_0\in H\oplus y$, we obtain $|D \tri (H \oplus x \oplus y \oplus z)| = O(K)$ as well.
\end{proof}

\begin{lemma}[Circular cubic]\label{lem:cir_elliptic}
Let $K>0$ and let $n$ be sufficiently large depending on $K$.
Suppose $P$ is a set of $n$ points in $\RR^2$ spanning at most $Kn^2$ ordinary generalised circles,
and all but at most $K$ points of $P$ lie on an acnodal or smooth circular cubic $\gamma$.  
Then there is a coset $H \oplus x$ of a subgroup $H\le \gamma^*$, 
with $4x \in H \oplus \omega$,
such that $|P \tri (H \oplus x)| = O(K)$.
\end{lemma}

\begin{proof}
Let $P' = P \cap \gamma^*$.
Then $|P\tri P'| = O(K)$, and by Lemma~\ref{lemma:stability}, $P'$ spans at most $O(Kn^2)$ ordinary circles. 
If $a$, $b$, $c \in \gamma$ are distinct, then by Proposition~\ref{prop:cir_elliptic}, the generalised circle through $a$, $b$, $c$ meets $\gamma$ again in the unique point $d=\omega\ominus(a\oplus b\oplus c)$. 
This implies that $d \in P'$ for all but at most $O(Kn^2)$ triples $a$, $b$, $c \in P'$,
or equivalently $a\oplus b\oplus c \in \omega \ominus P'$. 
Applying Corollary~\ref{cor:four} with $A = B = C = P'$ and $D = \omega\ominus P'$, we obtain $H\le\gamma^*$ and a coset $H \oplus x$ such that $|P \tri (H \oplus x)| = O(K)$ and 
$|(\omega\ominus P')\tri (H\oplus 3x)|=O(K)$, which is equivalent to $|P \tri (H \ominus 3x \oplus \omega)| = O(K)$. 
Thus we have $|(H \oplus x) \tri (H \ominus 3x \oplus \omega)| = O(K)$, which implies $4x \in H \oplus \omega$.
\end{proof}

\begin{lemma}[Concentric circles]\label{lem:2cir}
Let $K>0$ and let $n$ be sufficiently large depending on $K$.
Suppose $P$ is a set of $n$ points in $\RR^2$ spanning at most $Kn^2$ ordinary generalised circles.
Suppose all but at most $K$ of the points of $P$ lie on two concentric circles, and that $P$ has $n/2 \pm O(K)$ points on each.
Then, up to similarity, $P$ differs in at most $O(K)$ points from an `aligned' or `offset' double polygon.
\end{lemma}

\begin{proof}
By scaling and rotating, we can assume that $P$ lies mostly on the two concentric circles $\sigma_1 = \{ e^{2\pi i t} : t \in [0, 1) \}$ and $\sigma_2 = \{ re^{-2\pi i t} : t \in [0, 1) \}$, $r>1$, which we gave a group structure in Section \ref{sec:groupsonother}.

Let $P_1= P\cap\sigma_1$ and $P_2=P\cap\sigma_2$.
Then $|P\tri(P_1\cup P_2)|=O(K)$, and by Lemma~\ref{lemma:stability}, $P_1\cup P_2$ spans at most $O(Kn^2)$ ordinary circles.
If $a,b\in \sigma_1$ and $c\in\sigma_2$ with $a\neq b$, then by Lemma~\ref{prop:double-polygon-concyclic}, the generalised circle through $a$, $b$, $c$ meets $\sigma_1 \cup \sigma_2$ again in the unique point $d=\ominus(a\oplus b\oplus c)$.
This implies $d \in P_2$ for all but at most $O(Kn^2)$ triples $(a,b,c)$ with $a,b\in P_1$ and $c \in P_2$. Applying Corollary~\ref{cor:four} with $A = B = P_1$, $C=P_2$ and $D = \ominus P_2$, we get cosets $H \oplus x$ and $H\oplus y$ of $\sigma_1 \cup \sigma_2$ such that $|P_1 \tri (H \oplus x)|, |P_2 \tri (H \oplus y)| = O(K)$ and $2x\oplus 2y\in H$, where $x\in\sigma_1$ and $y\in\sigma_2$.
It follows that $H\le\sigma_1$, hence $H$ is a cyclic group of order $m=n/2\pm O(K)$, and $H\oplus x$ and $H\oplus y$ are the vertex sets of regular $m$-gons inscribed in $\sigma_1$ and $\sigma_2$, respectively, either `aligned' or `offset' depending on whether $x\oplus y\in H$ or not.
\end{proof}
Together these lemmas prove Theorem \ref{thm:strong}.
It just remains to remark that if $P$ differs in $O(K)$ points from a coset on an acnodal circular cubic, 
then we apply inversion in its singularity.
By Proposition \ref{prop:ellipseandacnodal},
we obtain that $P$ differs in $O(K)$ points from a coset $H\oplus x$ of a finite subgroup $H$ of an ellipse, where $4x=o$.
Thus, $x$ is a point of the ellipse with eccentric angle a multiple of $\pi/2$.
After a rotation, we can assume that $x=o$, which is Case~\ref{caseii} of Theorem \ref{thm:strong}.
\hfill$\blacksquare$


\section{Extremal configurations}\label{sec:extremal}
In this section we prove Theorems~\ref{thm:main}, \ref{thm:ordgencircles}, and \ref{thm:orchard}.
We first consider generalised circles.

\subsection{Ordinary generalised circles}

Suppose $P$ is an $n$-point set in $\RR^2$ spanning fewer than $\frac{1}{2} n^2$ ordinary generalised circles, and that $P$ is not contained in a generalised circle. 
Applying Theorem~\ref{thm:strong}, we can conclude that, up to inversions, $P$ differs in $O(1)$ points from one of the following examples: points on a line, 
a coset of a subgroup of an acnodal or smooth circular cubic, 
or a double polygon.

The first type of set is very easy to handle.
Note that the lower bound is on the number of ordinary circles, not counting $3$-point lines.
\begin{lemma}\label{lem:extremal_line}
Let $K\ge 1$ and $n\ge 2K+4$. If all except $K$ points of a set $P \subset \RR^2$ of $n$ points lie on a line, then $P$ spans at least $\binom{n-1}{2}$ ordinary circles.
\end{lemma}
\begin{proof}
Let $\ell$ be a line such that $|P\cap\ell|=n-K$.
For any $p\in P\cap\ell$ and $q\in P\setminus\ell$ there are at most $K-1$ non-ordinary circles through $p$, $q$, another point on $P\cap\ell$, and another point in $P\setminus \ell$.
Therefore, there are at least $K(n-2K)$ ordinary circles through $p$.
This holds for any of the $n-K$ points $p\in P\cap\ell$, and we obtain at least $\frac12 K(n-2K)(n-K)$ ordinary circles.
It is easy to see that when $1\le K\le (n-4)/2$, $\frac12 K(n-2K)(n-K)$ is minimised when $K=1$.
\end{proof}
Cosets on cubics are also relatively easy to handle.
We again obtain a lower bound on the number of ordinary circles, not including $3$-point lines.

\begin{lemma}\label{lem:extremal_cubic}
Suppose $P \subset \RR^2$ differs in $K$ points from a coset $H\oplus x$ of an acnodal or smooth circular cubic, where $|H|=n\pm O(K)$ and $4x\ominus\omega\in H$.
Then $P$ spans at least $\frac{1}{2}n^2 - O(Kn)$ ordinary circles.
\end{lemma}
\begin{proof}
Suppose that $P$ differs in $K$ points from $H\oplus x$.
We know from Construction~\ref{constr:cubic} that $H\oplus x$ spans $\frac{1}{2} n^2 - O(n)$ ordinary circles, all of which are tangent to $\gamma$. 
We show that adding or removing $K$ points destroys no more than $O(Kn)$ of these ordinary circles, so that the resulting set $P$ still spans at least $\frac{1}{2}n^2 - O(Kn)$ ordinary circles.

Suppose we add a point $q \notin H\oplus x$.
For $p \in H\oplus x$, at most one circle tangent to $\gamma$ at $p$ can pass through $q$. 
Thus, adding $q$ destroys at most $n$ ordinary circles. 
Now suppose we remove a point $p \in H\oplus x$. 
Since ordinary circles of $H\oplus x$ correspond to solutions of $2p\oplus q\oplus r = \omega$ or $p\oplus 2q\oplus r = \omega$, 
there are at most $O(n)$ solutions for a fixed $p$.
Thus removing $p$ destroys at most $O(n)$ ordinary circles.

Repeating $K$ times, we see that adding or removing $K$ points to or from $H\oplus x$ destroys at most $O(Kn)$ ordinary generalised circles out of the $\frac{1}{2}n^2 - O(n)$ spanned by $H\oplus x$.
This proves that $P$ spans at least $\frac{1}{2}n^2 - O(Kn)$ ordinary circles.
\end{proof}

From the two lemmas above we know that there is an absolute constant $C$ such that a set of $n$ points, not all collinear or concyclic, spanning at most $\frac{1}{2}n^2 - Cn$ ordinary generalised circles, differs in $O(1)$ points from Case~\ref{caseiv} in Theorem~\ref{thm:strong}. 
This case, 
where $P$ is close to the vertex set of a double polygon, 
requires a more careful analysis of the effect of adding or removing points.

We use the following special case of a result due to Raz, Sharir, and De Zeeuw \cite{RSZ15}.

\begin{prop}\label{prop:rsz}
If $P\subset \RR^2$ is a set of $n$ points contained in two circles,
then the number of lines with at least three points of $P$ is at most $O(n^{11/6})$.
\end{prop}
\begin{proof}
Denote the two circles by $\sigma_1$ and $\sigma_2$.
We use \cite{RSZ15}*{Theorem 6.1}, which states that for (not necessarily distinct) algebraic curves $C_1,C_2,C_3$ of constant degree, and finite sets $S_i\subset C_i$, 
the number of collinear triples $(p_1,p_2,p_3)\in S_1\times S_2\times S_3$, with $p_1,p_2,p_3$ distinct, is bounded by $O(|S_1|^{1/2}|S_2|^{2/3}|S_3|^{2/3}+|S_1| + |S_1|^{1/2}|S_2|+|S_1|^{1/2}|S_3|)$,
unless $C_1\cup C_2\cup C_3$ is a line or a cubic.
Let $C_1=\sigma_1$ and $C_2=C_3=\sigma_2$. 
Set $S_i = P\cap C_i$ for $i=1,2,3$.
Every line with at least one point of $S_1$ and two points of $S_2=S_3$ corresponds to a collinear triple in $S_1\times S_2\times S_3$.
Since the union of two circles is not a line or a cubic, 
we can apply the theorem to get the bound $O(n^{11/6})$ for the number of collinear triples in $P$ with one point in $\sigma_1$ and two points in $\sigma_2$.
Similarly, the number of collinear triples in $P$ with one point in $\sigma_2$ and two points in $\sigma_1$ is also $O(n^{11/6})$.
Since a line intersects $\sigma_1\cup\sigma_2$ in at most four points, we also obtain the bound $O(n^{11/6})$ for the number of lines with at least three points.
\end{proof}

\begin{lemma}\label{lem:add/remove}
Let $S$ be a double polygon with $m$
points on each circle. 
Let $P = (S \setminus A) \cup B$ be a set of $n$ points, where $A$ is a subset of $S$ with $a = O(1)$ points and $B$ is a set disjoint from $S$ with $b = O(1)$ points. 
Then $P$ spans at least $\frac18 (2 + a + 4b)n^2 - O(n^{11/6})$ ordinary generalised circles.
\end{lemma}
\begin{proof}
We know from Constructions~\ref{constr:even} and~\ref{constr:offset} that $S$ spans $\frac14 n^2 - O(n)$ ordinary generalised circles.

Consider first the number of ordinary generalised circles spanned by $S \setminus A$. 
As we saw in Construction~\ref{constr:odd}, 
removing a point $p \in S$ destroys at most $3m/2$ ordinary generalised circles spanned by $S$, and adds $\frac12 m^2 - O(m) = \frac18 n^2 - O(n)$ ordinary generalised circles. 
Noting that there are at most $m$ $4$-point generalised circles spanned by $S$ that go through any two given points of $A$, we thus have by inclusion-exclusion that $S \setminus A$ determines at least $(\frac14+\frac{a}{8})n^2 - O(n)$ ordinary generalised circles.

Now consider adding $q\in B$ to $S$. 
For any pair of points from $S\setminus A$,
adding $q\in B$ creates a new ordinary generalised circle,
unless the generalised circle through the pair and $q$ contains three or four points of $S\setminus A$.
We already saw that the number of ordinary generalised circles hitting a fixed point is $O(n)$, so it remains to bound the number of $4$-point generalised circles of $S$ that hit $q$.
If $q$ lies on one of the concentric circles, then no $4$-point generalised circles hit $q$, so we can assume that $q$ does not.
Applying inversion in $q$ reduces the problem to bounding the number of $4$-point lines determined by a subset of two circles.
By Proposition \ref{prop:rsz},
this number is bounded by $O(n^{11/6})$,
so $p$ lies on at most $O(n^{11/6})$ of the $4$-point generalised circles spanned by $S$.
Adding $q$ to $S$ thus creates at least $\binom{n}{2} - O(n^{11/6})$ ordinary generalised circles. 
Note that each $p \in A$ that was removed destroys at most $n$ of these circles. 

Adding $q$ to $S\setminus A$ also destroys at most $O(n)$ ordinary circles, 
since for each $p\in S$ there is only one circle tangent at $p$ and going through $q$, 
and for each $p \in A$, at most $m$ ordinary circles spanned by $S \setminus A$ go through $p$. 
Finally, since there are at most $2m$ circles through two points of $B$ that also go through two points of $S \setminus A$, $P = (S \setminus A) \cup B$ spans at least $(\frac14+\frac{a}{8}+\frac{b}{2})n^2 - O(n^{11/6})$ ordinary generalised circles.
\end{proof}

Theorem~\ref{thm:ordgencircles} then follows easily from the lemmas above.

\begin{proof}[Proof of Theorem~\ref{thm:ordgencircles}]
Suppose that $P$ is a set of $n$ points in $\RR^2$ with fewer than $\frac12n^2-Cn$ ordinary generalised circles, where $C$ is sufficiently large.
Without loss of generality, $n$ is also sufficiently large.
By Lemmas~\ref{lem:extremal_line} and \ref{lem:extremal_cubic}, 
we need only consider the case where $P$ differs by $O(1)$ points from a double polygon.
In the notation of Lemma~\ref{lem:add/remove}, 
we have $P = (S \setminus A) \cup B$ and $\frac{1}{8} (2+a+4b) < \frac12$, which implies that $a \le 1$ and $b = 0$. 
So $P$ is either equal to $S$, or is obtained from $S$ by removing one point, 
which are exactly the cases in Constructions~\ref{constr:even}, \ref{constr:offset}, and~\ref{constr:odd}.
In particular, the minimum number of ordinary generalised circles occurs in Construction~\ref{constr:even} when $n\equiv0\pmod{4}$, in Construction~\ref{constr:odd} when $n\equiv1,3\pmod{4}$, and in Constructions~\ref{constr:even} and \ref{constr:offset} when $n\equiv2\pmod{4}$.
\renewcommand{\qedsymbol}{$\blacksquare$}
\end{proof}


\subsection{Ordinary circles}\label{sec:circles}

We now consider what happens if we do not count generalised circles that are lines, 
and prove Theorem~\ref{thm:main}.
\begin{proof}[Proof of Theorem~\ref{thm:main}]
Let $P$ be a set of $n$ points not all on a line or a circle, with at most $\frac12 n^2-Cn$ ordinary circles, for a sufficiently large $C$.
By a simple double counting argument, there are at most $\frac{1}{6}n^2$ $3$-point lines,
so there are at most $\frac23n^2 - O(n)$ ordinary generalised circles.
By Theorem~\ref{thm:strong}, up to inversions and up to $O(1)$ points, $P$ lies on a line, an ellipse, a smooth circular cubic, or two concentric circles.
By Lemmas~\ref{lem:extremal_line} and \ref{lem:extremal_cubic}, the first three cases give us at least $\frac12n^2-O(n)$ ordinary circles, contrary to assumption.
Therefore, 
we only need to consider the case where, when $P$ is transformed by an inversion to $P'$,
we have $P'=(S\setminus A)\cup B$,
where $S$ is a double polygon (`aligned' or `offset'),
and $|A|=a$, $|B|=b$.

By Lemma~\ref{lem:add/remove}, $P'$ has at least $\frac18(2+a+4b)n^2-O(n^{11/6})$ ordinary generalised circles, which gives us the inequality $\frac18 (2+a+4b) < \frac23$, 
which in turn gives us $a\le 3$ and $b=0$.
Therefore, $P'$ lies on two concentric circles, and $P$ lies on the disjoint union of two circles or the disjoint union of a line and a circle.

Suppose that $a=3$ (and $b=0$).
Then $P'$ has $\frac58n^2-O(n)$ ordinary generalised circles.
Those passing through the centre of the inversion that transforms $P$ to $P'$, are inverted back to straight lines passing through three points of $P$.
As in the proof of Lemma~\ref{lem:add/remove}, there are $\frac18n^2 - O(n)$ ordinary generalised circles that pass through any point of $A$.
Also, we can use Lemma~\ref{lemma:tangent} below to show that there are at most $O(n)$ ordinary generalised circles spanned by $S\setminus A$ that intersect in the same point not in $S$.
Indeed, by Lemma~\ref{lemma:tangent}, there are at most $n/2$ ordinary generalised circles of $S$ that intersect in the same point $p\notin S$.
Furthermore, for each point $q\in A$ there are $O(n)$ generalised circles through $p$, $q$, and two more points of $S$.
It follows that there are $O(n)$ ordinary generalised circles spanned by $S\setminus A$ through $p$.

Thus, if the centre of inversion is in $A$, $P$ has $\frac12n^2-O(n)$ ordinary circles, which is a contradiction if $C$ is chosen large enough.
On the other hand, if the centre of inversion is not in $A$, then $P$ has $\frac58n^2-O(n)$ ordinary circles, also a contradiction.

Therefore, we have $a\le 2$, which means that $P'$ is a set of $n$ points as in Constructions~\ref{constr:even}, \ref{constr:offset}, \ref{constr:odd}, or \ref{constr:linecircle}.

Next, suppose that $n$ is even.
If $a=2$, then there are $\frac12n^2-O(n)$ ordinary generalised circles and through both points of $A$ there are $\frac18n^2 - O(n)$ ordinary generalised circles.
If we invert in one of these points in $A$, we obtain a set with $\frac38n^2-O(n)$ ordinary circles (as in Construction~\ref{constr:linecircle}), which is not extremal.
Otherwise, $a=0$, $P'$ is as in Constructions~\ref{constr:even} or \ref{constr:offset}, and there are at least $\frac14n^2-n$ ordinary generalised circles if $n\equiv0\pmod{4}$ and $\frac14n^2-\frac12 n$ if $n\equiv2\pmod{4}$.
Let $p$ be the centre of the inversion that transforms $P$ to $P'$.
Then all the $3$-point lines of $P$ are inverted to ordinary circles in the double polygon $P'$, all passing through $p$.
By Lemma~\ref{lemma:tangent} below, there are at most $n/2$ ordinary circles that intersect in the same point not in $P'$.
Thus, in $P$ there at most $n/2$ $3$-point lines,
and the number of ordinary circles (not including lines) is at least $\frac14n^2-\frac32n$ if $n\equiv0\pmod{4}$ and $\frac14n^2-n$ if $n\equiv2\pmod{4}$, which match
Construction~\ref{constr:even} (and Construction~\ref{constr:offset} if $n\equiv2\pmod{4}$), if the radii are chosen so that each vertex of the inner polygon has an ordinary generalised circle that is a straight line tangent to it.

Finally, suppose that $n$ is odd.
Then $a=1$ and $P'$ is as in Construction~\ref{constr:odd}, with $\frac38n^2 - O(n)$ ordinary generalised circles.
It follows that $P$ must be as in Construction~\ref{constr:linecircle}, with $\frac{1}{4}n^2-\frac{3}{4}n+\frac{1}{2}$  ordinary circles if $n\equiv 1\pmod{4}$ and $\frac{1}{4}n^2-\frac{5}{4}n+\frac{3}{2}$ ordinary circles if $n\equiv 3\pmod{4}$.
This finishes the proof.
\renewcommand{\qedsymbol}{$\blacksquare$}
\end{proof}

\begin{lemma}\label{lemma:tangent}
Let $S$ be a double polygon (`aligned' or `offset') with $m$ points on each circle. Then a point $q \notin S$ lies on at most $m$ ordinary generalised circles spanned by $S$.
\end{lemma}

\begin{proof}
Denote the inner circle by $\sigma_1$ and the outer circle by $\sigma_2$, both with centre $o$.
We proceed by case analysis on the position of $q$ with respect to $\sigma_1$ and $\sigma_2$. Note that for each point $p \in S$, at most one of the ordinary generalised circles tangent at $p$ can go through $q$. 

If $q$ lies on either $\sigma_1$ or $\sigma_2$, then $q$ does not lie on any ordinary generalised circle spanned by $S$.

If $q$ lies inside $\sigma_1$, then $q$ lies on at most $m$ ordinary generalised circles spanned by $S$, since ordinary generalised circles tangent to $\sigma_1$ cannot pass through $q$.
Similarly, if $q$ lies outside $\sigma_2$, it lies on at most $m$ ordinary generalised circles, since ordinary generalised circles tangent to $\sigma_2$ lie inside $\sigma_2$.

The remaining case to consider is when $q$ lies in the annulus bounded by $\sigma_1$ and $\sigma_2$.
Consider the subset $S' \subset S$ of points $p$ such that there exists an ordinary generalised circle tangent at $p$ going through $q$.
Consider the four circles passing through $q$ and tangent to both $\sigma_1$ and $\sigma_2$.
They touch $\sigma_1$ at $a_1,b_1,c_1,d_1$ and $\sigma_2$ at $a_2,b_2,c_2,d_2$ as in Figure~\ref{fig:bitangent_circles}.
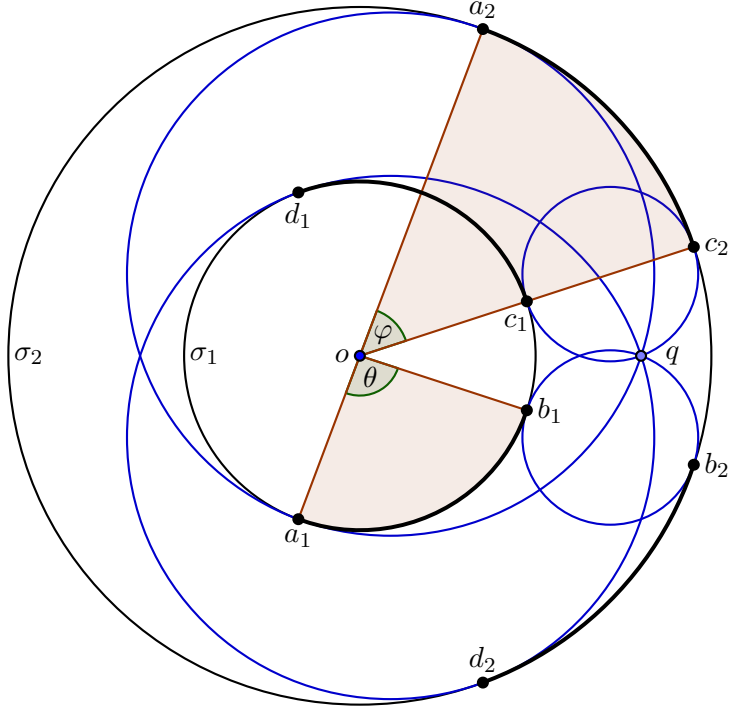
\begin{figure}
\centering
\definecolor{qqwuqq}{rgb}{0.,0.39215686274509803,0.}
\definecolor{zzttqq}{rgb}{0.6,0.2,0.}
\definecolor{uuuuuu}{rgb}{0,0,0}
\definecolor{qqqqcc}{rgb}{0.,0.,0.8}
\definecolor{xdxdff}{rgb}{0.49019607843137253,0.49019607843137253,1.}
\definecolor{yqyqyq}{rgb}{0.5019607843137255,0.5019607843137255,0.5019607843137255}
\definecolor{qqqqff}{rgb}{0.,0.,1.}
\begin{tikzpicture}[line cap=round,line join=round,>=triangle 45,thick,x=1.0cm,y=1.0cm,scale=0.925]
\draw [shift={(5.,0.5)},color=qqwuqq,fill=qqwuqq,fill opacity=0.1] (0,0) -- (18.194872338766768:0.69) arc (18.194872338766768:69.51268488527734:0.69) -- cycle;
\draw [shift={(5.,0.5)},color=qqwuqq,fill=qqwuqq,fill opacity=0.1] (0,0) -- (-110.48731511472266:0.57) arc (-110.48731511472266:-18.194872338766785:0.57) -- cycle;
\draw [color=uuuuuu] (5.,0.5) circle (2.5cm);
\draw [color=uuuuuu] (5.,0.5) circle (5.cm);
\draw [color=qqqqcc] (5.4375,1.6709371246996996) circle (3.75cm);
\draw [color=qqqqcc] (5.4375,-0.6709371246996997) circle (3.75cm);
\draw [color=qqqqcc] (8.5625,1.6709371246996996) circle (1.25cm);
\draw [color=qqqqcc] (8.5625,-0.6709371246996997) circle (1.25cm);
\draw [shift={(5.,0.5)},color=zzttqq,fill=zzttqq,fill opacity=0.1]  (0,0) --  plot[domain=0.31756042929152123:1.213225223149386,variable=\t]({1.*5.*cos(\t r)+0.*5.*sin(\t r)},{0.*5.*cos(\t r)+1.*5.*sin(\t r)}) -- cycle ;
\draw [shift={(5.,0.5)},color=zzttqq,fill=zzttqq,fill opacity=0.1]  (0,0) --  plot[domain=4.35481787673918:5.9656248778880645,variable=\t]({1.*2.5*cos(\t r)+0.*2.5*sin(\t r)},{0.*2.5*cos(\t r)+1.*2.5*sin(\t r)}) -- cycle ;
\draw [shift={(5.,0.5)},line width=1.5pt]  plot[domain=0.31756042929152123:1.213225223149386,variable=\t]({1.*5.*cos(\t r)+0.*5.*sin(\t r)},{0.*5.*cos(\t r)+1.*5.*sin(\t r)});
\draw [shift={(5.,0.5)},line width=1.5pt]  plot[domain=4.35481787673918:5.9656248778880645,variable=\t]({1.*2.5*cos(\t r)+0.*2.5*sin(\t r)},{0.*2.5*cos(\t r)+1.*2.5*sin(\t r)});
\draw [shift={(5.,0.5)},line width=1.5pt]  plot[domain=0.31756042929152134:1.9283674304404066,variable=\t]({1.*2.5*cos(\t r)+0.*2.5*sin(\t r)},{0.*2.5*cos(\t r)+1.*2.5*sin(\t r)});
\draw [shift={(5.,0.5)},line width=1.5pt]  plot[domain=5.0699600840302:5.9656248778880645,variable=\t]({1.*5.*cos(\t r)+0.*5.*sin(\t r)},{0.*5.*cos(\t r)+1.*5.*sin(\t r)});
\draw [fill=qqqqff] (5.,0.5) circle (2pt) node[left] {$o$};
\draw [fill=xdxdff] (9.,0.5) circle (2pt);
\draw[color=uuuuuu] (9.45,0.5) node {$q$};
\draw [fill=uuuuuu] (4.125,-1.8418742493994003) circle (2pt) node[below] {$a_1$};
\draw [fill=uuuuuu] (7.375,-0.2806247497997997) circle (2pt) node[right] {$b_1$};
\draw [fill=uuuuuu] (7.375,1.2806247497998) circle (2pt);
\draw [fill=uuuuuu] (7.23,0.97) node {$c_1$};
\draw [fill=uuuuuu] (4.125,2.8418742493993996) circle (2pt) node[below] {$d_1$};
\draw [fill=uuuuuu] (6.75,-4.183748498798798) circle (2pt) node[above] {$d_2$};
\draw [fill=uuuuuu] (9.75,-1.0612494995995996) circle (2pt) node[right] {$b_2$};
\draw [fill=uuuuuu] (9.75,2.0612494995995996) circle (2pt) node[right] {$c_2$};
\draw [fill=uuuuuu] (6.75,5.183748498798798) circle (2pt) node[above] {$a_2$};
\draw[color=uuuuuu] (5.33,0.8) node {$\varphi$};
\draw[color=uuuuuu] (5.15,0.2) node {$\theta$};
\draw (2.79,0.5) node {$\sigma_1$};
\draw (0.29,0.5) node {$\sigma_2$};
\end{tikzpicture}
\caption{Bitangent circles through $q$}\label{fig:bitangent_circles}
\end{figure}
Any circle through $q$ tangent to $\sigma_1$ and intersecting $\sigma_2$ in two points, must touch $\sigma_1$ on one of the open arcs $a_1b_1$ or $c_1d_1$.
Similarly, 
any circle through $q$ tangent to $\sigma_2$ and intersecting $\sigma_1$ in two points, must touch $\sigma_2$ on one of the open arcs $a_2c_2$ or $b_2d_2$.
It follows that $S'$ must be contained in the relative interiors of one of these four arcs.
Since $S$ consists of $m$ equally spaced points on each of $\sigma_1$ and $\sigma_2$,
\[|S'|< \left\lceil \frac{2m (\angle a_1ob_1 + \angle c_1od_1 + \angle b_2od_2 + \angle a_2oc_2)}{4\pi} \right\rceil = \left\lceil \frac{m(\theta+\phi)}{\pi}\right\rceil,\]
where $\theta$ and $\phi$ are as indicated in Figure~\ref{fig:bitangent_circles}.
In order to show that $|S'|\le m$, it suffices to show that the angle sum $\theta + \phi$ is strictly less than $\pi$.
This is clear from Figure~\ref{fig:bitangent_circles} (note that $a_1, o, a_2$ are collinear with $a_1$ and $a_2$ on opposite sides of $o$).
\end{proof}


\subsection{Four-point circles}

\begin{proof}[Proof of Theorem~\ref{thm:orchard}]
Let $P$ be a set of $n$ points in $\RR^2$ with at least $\frac{1}{24}n^3 - \frac{7}{24} n^2 + O(n)$ $4$-point generalised circles.
Let $t_i$ denote the number of $i$-point lines ($i\ge 2$) and $s_i$ the number of $i$-point circles ($i\ge 3$) in $P$.
By counting unordered triples of points, we have
\[ \binom{n}{3} = \sum_{i\ge 3}\binom{i}{3}(t_i+s_i) \ge t_3+s_3+4(t_4+s_4),\] hence \[\frac{1}{6}n^3 - O(n^2) \ge t_3 + s_3 + 4\left(\frac{1}{24}n^3 - O(n^2)\right)\] and $t_3+s_3 = O(n^2)$, so we can apply Theorem~\ref{thm:strong}.
We next consider each of the cases of that theorem in turn.

If all except $O(1)$ points of $P$ lie on a straight line, it is easy to see that $P$ determines only $O(n^2)$ generalised circles, contrary to assumption.

If all except $O(1)$ are vertices of two regular $m$-gons on concentric circles where $m= n/2 \pm O(1)$, then we know from Constructions~\ref{constr:even}, \ref{constr:offset}, and~\ref{constr:odd} that $P$ determines at most $\frac{1}{32}n^3 +O(n^2)$ $4$-point generalised circles, again contrary to assumption.

Suppose next that $P = ((H \oplus x)\setminus A) \cup B$, where $H$ is a finite subgroup of order $m=n\pm O(1)$ of a smooth circular cubic, $A$ is a subset of $H \oplus x$ with $a=O(1)$ points, and $B$ is a set disjoint from $H \oplus x$ with $b=O(1)$ points.
Then $n=m-a+b$.
The number of $4$-point generalised circles in $H \oplus x$ is $\frac{1}{24}m^3-\frac{1}{4}m^2 + O(m)$.
We next determine an upper bound for the number of $4$-point generalised circles in $P$.

For each $p\in A$, let $C_p$ be the set of $4$-point generalised circles of $H \oplus x$ that pass through $p$.
Then $|C_p|=\frac16 m^2 - O(m)$ and $|C_p\cap C_q|=O(m)$ for distinct $p,q\in A$.
By inclusion-exclusion, we destroy at least $|\bigcup_{p\in A}C_p|\ge \frac16 am^2 - O(m)$ $4$-point generalised circles by removing $A$, and we still have at most $\frac{1}{24}m^3-\frac{1}{4}m^2 - \frac16 am^2 + O(m)$ $4$-point generalised circles in $(H\oplus x)\setminus A$.

For each $p\in B$, the number of ordinary generalised circles spanned by $H \oplus x$ passing through $p$ is at most $O(m)$.
This is because each such generalised circle is tangent to the cubic at one of the points of $H \oplus x$, and there is only one generalised circle through $p$ and tangent at a given point of $H \oplus x$.
Also, for each pair of distinct $p,q\in B$, there are at most $O(m)$ generalised circles through $p$ and $q$ and two points of $H \oplus x$; and for any three $p,q,r\in B$ there are at most $O(1)$ generalised circles through $p,q,r$ and one point of $H \oplus x$.
Therefore, again by inclusion-exclusion, by adding $B$ we gain at most $O(m)$ $4$-point generalised circles.

It follows that the number of $4$-point generalised circles determined by $P$ is 
\[ t_4+s_4 \le \frac{1}{24}m^3-\frac{1}{4}m^2 - \frac16 am^2 + O(m) = \frac{n^3 - (a+3b+6)n^2 +O(n)}{24}.\]
Since we assumed that
\[t_4+s_4\ge \frac{n^3-7n^2+O(n)}{24},\] we obtain $a+3b < 1$.
Therefore, $a=b=0$ and $P= H \oplus x$.
The maximum number of $4$-point circles in a coset has been determined in Constructions~\ref{constr:ellipse} and \ref{constr:cubic}.

The final case, when all but $O(1)$ points of $P$ lie on an ellipse,
can be reduced to the previous case.
Indeed, by Lemma~\ref{prop:ellipseandacnodal}, if we invert the ellipse in a point on the ellipse, we obtain an acnodal circular cubic, and then the above analysis holds verbatim for the group of regular points on this cubic.
\renewcommand{\qedsymbol}{$\blacksquare$}
\end{proof}



\end{document}